\documentclass[a4paper,11pt]{article}
\usepackage{mathrsfs}
\usepackage{latexsym}
\usepackage{amsmath,amssymb}
\usepackage{amsthm}
\usepackage{amsfonts}
\usepackage[usenames]{color}
\usepackage{amssymb}
\usepackage{graphicx}
\usepackage{amsmath}
\usepackage{amsfonts}
\usepackage{amsthm}
\usepackage{mathrsfs}
\usepackage{dsfont}
\usepackage{indentfirst}
\usepackage{enumerate}
\usepackage{enumitem}
\usepackage{esint}
\usepackage{bbm}
\usepackage[square, comma, sort&compress, numbers]{natbib}

\newtheoremstyle{mythm}{1.5ex plus 1ex minus .2ex}{1.5ex plus 1ex
minus .2ex}{\kai}{\parindent}{\song\bfseries}{}{1em}{}
\numberwithin{equation}{section}

\newtheorem{thm}{Theorem}[section]
\newtheorem{lemma}{Lemma}[section]
\newtheorem{proposition}{Proposition}[section]
\newtheorem{remark}{Remark}[section]
\newtheorem{corollary}{Corollary}[section]
\allowdisplaybreaks[4]
\newcommand\leqs{\leqslant}
\newcommand\geqs{\geqslant}

\makeatletter
\newcommand\Rmnum[1]{\expandafter\@slowromancap\romannumeral #1@}
\makeatother

\title{Gradient Estimates for Hessian Equations with Measure Data}
\author{Rongxun He}

\begin{document}
\date{}
\maketitle

\begin{abstract}
In this article, we establish the interior gradient estimates for $k$-Hessian equations with measure data. We use the ideas from the potential theories for the $p$-Laplacian type equations and pointwise estimates of $k$-subharmonic functions by Labutin. The interior gradient regularity is optimal for $k\geqs3$, while for $k=2$, we prove the interior $W^{2,1+\varepsilon}$ and $C^{1,\alpha}$ estimates for $2$-subharmonic solutions.
\end{abstract}

\section{Introduction}
The notion of $k$-Hessian measure and $k$-subharmonic function in $\mathbb{R}^n$ was introduced by Trudinger and Wang in \cite{TW97measureI}, \cite{TW97measureII} and \cite{TW97measureIII}. Their results provided a framework for the theory of weak solutions of $k$-Hessian equations, extending the classical theory of Riesz measure when $k=1$ and Monge–Ampère measure when $k=n$. Later in \cite{Labutin02potential}, Labutin studied the potential theory for $k$-subharmonic functions analogous to the classical potential theory for subharmonic functions. In this article, we concern the further properties for such functions. More precisely, we establish the interior gradient estimates by means of the natural Wolff potential of certain $k$-Hessian measure. The resulting estimates provide a comprehensive picture concerning the pure interior regularity results for $k$-Hessian equations. Our work is inspired by the gradient potential theory of quasilinear equations (see \cite{DM10potential, DM11potential, DZ24potential, KM13potential, Mingione11potential, NP20potential, NP23potential}); however, some of our results are totally different, since the key estimates in the quasilinear case are not valid for the $k$-Hessian operator.

Let $\Omega\subset\mathbb{R}^n$ and $u\in C^2(\Omega)$. The $k$-Hessian operator $S_k$ is defined by
\begin{align*}
S_k[u]=\sigma_k(\lambda(D^2u))=\sum_{i_1<\cdots<i_k}\lambda_{i_1}\cdots\lambda_{i_k},
\end{align*}
where $\lambda(D^2u)=(\lambda_1,\ldots,\lambda_n)$ are the eigenvalues of Hessian matrix $D^2u$, and $\sigma_k(\lambda)$ denotes the $k$-th elementary symmetric polynomial of $\lambda$. In particular, this is the Laplace operator and the Monge-Ampère operator as $k=1$ and $k=n$, respectively. To work in the realm of elliptic operator, we consider the equation on the standard G\aa rding cone $\Gamma_k$,
\begin{align*}
\Gamma_k=\{\lambda\in\mathbb{R}^n:\sigma_j(\lambda)>0,j=1,\ldots,k\}.
\end{align*}
Clearly, $\Gamma_n\subset\cdots\subset\Gamma_1$. To formulate admissibility in a weak sense, we adopt the definition based on viscosity solutions, see \cite{TW97measureII} and \cite{Labutin02potential}. An upper semi-continuous function $u:\Omega\to[-\infty,+\infty)$, not identically $-\infty$ on any component of $\Omega$, is called $k$-\textit{subharmonic} in $\Omega$ if $S_k[q]\geqs0$ for any quadratic polynomial $q$ for which $u-q$ has a finite local maximum in $\Omega$. We denote by $\Phi^k(\Omega)$ the set of all $k$-subharmonic functions in $\Omega$, and for $u\in\Phi^k(\Omega)$ we denote by $\mu_k(u)$ the $k$-Hessian measure of $u$ (see \cite{TW97measureI}, \cite{TW97measureII} and \cite{TW97measureIII}). It is a Borel measure which coincides with the usual density for smooth functions,
\begin{align*}
\mu_k(u)=S_k[u]dx\quad\text{for }u\in C^2(\Omega).
\end{align*}
It was proved in \cite{TW97measureII} that $k$-Hessian measure is weakly continuous with respect to the local $L^1$-convergence. Hence, it is admissible to approximate $k$-subharmonic functions by a decreasing sequence of smooth functions in $\Phi^k(\Omega)$, and the limiting measure is independent of the particular approximating sequence.

For general $k$-subharmonic function $u\leqs0$, the local gradient estimate holds, see \cite{TW97measureII},
\begin{align*}
\Vert Du\Vert_{L^q(\Omega')}\leqs C\int_{\Omega}|u|,
\end{align*}
where $\Omega'\Subset\Omega$ and $q<\frac{nk}{n-k}$. This implies that $\Phi^k(\Omega)\subset W_{loc}^{1,q}(\Omega)$ for $q<\frac{nk}{n-k}$. Moreover, we have the inclusion $\Phi^k(\Omega)\subset C_{loc}^{0,\alpha}(\Omega)$ for $\alpha=2-\frac{n}{k}$ and $k>\frac{n}{2}$, see \cite{TW97measureI}.

When turning to the case $1\leqs k\leqs\frac{n}{2}$, Trudinger\,\cite{Trudinger97weak} first derived the local H\"older estimate for weak admissible solution to $S_k[u]=f$ with $f\in L^p(\Omega)$ for $p>\frac{n}{2k}$. For general $k$-subharmonic functions, Labutin\,\cite{Labutin02potential} proved a pointwise estimate in terms of the Wolff type potential
\begin{align*}
\mathbf{W}^\mu_{\frac{2k}{k+1},k+1}(x,r)=\int_0^r\Big(\frac{\mu(B(x,t))}{t^{n-2k}}\Big)^{1/k}\frac{dt}{t}.
\end{align*}
More precisely, if $u\in\Phi^k(\Omega)$ satisfies $\mu_k(u)=\mu$ and $u\leqs0$, then
\begin{align}
C_1\mathbf{W}^\mu_{\frac{2k}{k+1},k+1}(x,R/8)\leqs|u(x)|\leqs C_2\mathbf{W}^\mu_{\frac{2k}{k+1},k+1}(x,2R)-C_3\sup_{B_R}u,
\label{formu11}
\end{align}
whenever $B(x,3R)\subset\Omega$. In particular, if $\mu(B(x,r))\leqs Mr^{n-2k+\varepsilon}$ holds for all $B(x,8r)\subset\Omega$ with some $M,\varepsilon>0$, then $u$ is locally $\alpha$-H\"older continuous in $\Omega$ with $\alpha=\alpha(n,k,\varepsilon)>0$.

In view of the above results, a natural question is whether we can establish the gradient estimates for $k$-subharmonic functions $u$, by making use of the Hessian measure $\mu_k(u)$. This question was first proposed by Trudinger\,\cite{Trudinger97weak} for $\mu_k(u)=fdx$ with $f\in L^p(\Omega)$, $p>\frac{n}{k}$. In this article, we settle the problem and obtain the optimal interior regularity results for $k$-subharmonic functions. For simplicity, we use the notation $\mathbf{W}^\mu_k(x,r)$ for the Wolff type potential with correct scaling,
\begin{align}
\mathbf{W}^\mu_k(x,r):=\mathbf{W}^\mu_{\frac{k}{k+1},k+1}(x,r)=\int_0^r\Big(\frac{\mu(B(x,t))}{t^{n-k}}\Big)^{1/k}\frac{dt}{t}.
\label{formu12}
\end{align}
\begin{thm}\label{thm11}
Let $2\leqs k\leqs n$, $u\in\Phi^k(\Omega)$ and $\mu_k(u)=\mu$. Then there exists a constant $C=C(n,k)>0$ such that
\begin{align}
|\nabla u(x_0)|\leqs C\fint_{B(x_0,R)}|\nabla u|+C\mathbf{W}_k^\mu(x_0,2R)
\label{formu13}
\end{align}
holds whenever $B(x_0,3R)\subset\Omega$. Moreover, if the Wolff potential $\mathbf{W}_k^\mu(x,r)$ is finite, then $u$ is locally Lipschitz in $\Omega$.
\end{thm}
It is known that $u\in\Phi^k(\Omega)\subset W_{loc}^{1,q}(\Omega)$ for $q<\frac{nk}{n-k}$. Accordingly, the gradient estimate in \eqref{formu13} is understood to hold almost everywhere in $\Omega$. In particular, the corresponding Lipschitz estimate holds at all points. The following corollary is a direct consequence by applying the formula \eqref{formu12} to Theorem \ref{thm11}.

\begin{corollary}\label{cor11}
Let $2\leqs k\leqs n$, $u\in\Phi^k(\Omega)$ and $\mu_k(u)=\mu$. If there exist constants $M>0$, $\varepsilon>0$ such that
\begin{align*}
\mu(B(x,r))\leqs Mr^{n-k+\varepsilon}
\end{align*}
for all $x,r$ such that $B(x,3r)\subset\Omega$, then $u$ is locally Lipschitz in $\Omega$.

In particular, if $\mu_k(u)=f dx$ with $f\in L^p_{loc}(\Omega)$, $p>\frac{n}{k}$, then $u$ is locally Lipschitz in $\Omega$.
\end{corollary}

The interior gradient regularity is in a sense optimal, as there exists a merely Lipschitz $k$-subharmonic function $u(x)=|x_1|$, which solves
\begin{align*}
\mu_k(u)=0\quad\text{for}\quad 2\leqs k\leqs n.
\end{align*}
For the nondegenerate case, we further have a counterexample for $3\leqs k\leqs n$:
\begin{align*}
u(x)=|x'|+(1+x_k^2)|x'|^{k/2},\quad\text{where}\quad x'=(x_1,\ldots,x_{k-1}).
\end{align*}
Solution of this type was founded by Caffarelli for $k=n$, see also \cite{DF14Lipschitz}. Then $u$ is a convex function and it is not $C^1$-differentiable at $x'=0$. Moreover, the corresponding Hessian measure has finite positive upper and lower bounds near the origin. Indeed, by direct calculations, one has
\begin{align*}
\mu_k(u)=k\Big[\big(\frac{k}{2}-1\big)-\big(\frac{k}{2}+1\big)x_k^2\Big]\Big[1+\frac{k}{2}|x'|^{\frac{k}{2}-1}(1+x_k^2)\Big]^{k-2}
\end{align*}
and $0<c_k\leqs \mu_k(u)\leqs C_k<\infty$ in $B_{1/4}$. Additionally, Caffarelli and Yuan\,\cite{CY22singular} constructed a merely Lipschitz convex solution $U_n$ to $\det D^2u=1$ for $n\geqs3$. Consequently, by setting $u(x)=U_k(x_1,\ldots,x_k)$, $x\in\mathbb{R}^n$ for $3\leqs k\leqs n$, we obtain a merely Lipschitz solution to $\mu_k(u)=1$. Both examples demonstrate the optimality of the interior (gradient) regularity results for $k$-Hessian equations when $k\geqs3$.
\\[-0.5em]

In the following, we focus on the case $k=2$. There have been extensive results on the interior regularity for the following quadratic Hessian equation,
\begin{align*}
S_2[u]=f\quad\text{in }\Omega.
\end{align*}
When $n=2$, it falls into the Monge-Ampère equation and the interior Hessian estimate for smooth convex solutions was established by Heinz\,\cite{Heinz59two}. For $f$ only with positive upper and lower bounds, Caffarelli\,\cite{Caffarelli91C1alpha} obtained the interior $C^{1,\alpha}$ regularity for Alexandrov solutions. When $n=3$, Warren and Yuan\,\cite{WY09sigma2} proved the interior $C^2$ estimates for constant $f=1$, and Qiu\,\cite{Qiu24sigma2} derived the corresponding results for positive $f\in C^{1,1}$. Shankar and Yuan\,\cite{SY25sigma2} resolved the problem of interior Hessian estimates for constant equation in dimension four, and later Fan\,\cite{Fan26sigma2} extended their results to positive $C^{1,1}$ densities. In arbitrary dimensions, many results have also been established for convex, almost convex or semiconvex solutions (see, e.g., \cite{GQ19scalar,MSY19sigma2,SY20sigma2,Mooney21sigma2,CJTZ26sigma2,ZZ26sigma2}). Recently, Li and Wu\,\cite{LW26sigma2} removed the various convexity and dimension restrictions and proved the interior Hessian estimates for constant $2$-Hessian equation; their method applies to positive $C^{1,1}$ right-hand sides as well. Chen, Zhou and Zhu\,\cite{CZZ26sigma2} later obtained the interior $C^{2,\alpha}$ regularity results for equation with positive $C^\alpha$ variable right sides.

Our result is as follows. In fact, we consider the quadratic Hessian equation with positive right-hand sides bounded above and below, and prove the interior $C^{1,\alpha}$ estimates for $2$-subharmonic solutions in arbitrary dimensions.

\begin{thm}\label{thm12}
Let $u\in\Phi^2(\Omega)$ and $\mu_2(u)=fdx$ with $f\in L^\infty(\Omega)$. If $f$ satisfies
\begin{align*}
0<\lambda\leqs f\leqs\Lambda<\infty\quad\text{in }\Omega,
\end{align*}
then there exists a constant $\alpha=\alpha(n,\lambda,\Lambda)\in(0,1)$ such that $u\in C^{1,\alpha}_{loc}(\Omega)$. In particular, if $B(x_0,4R)\subset\Omega$, then it holds
\begin{align*}
[Du]_{C^\alpha(B(x_0,R))}\leqs C(n,\lambda,\Lambda,\alpha ,R,\Vert u\Vert_{L^\infty(B(x_0,4R))}).
\end{align*}
\end{thm}

The proof of Theorem \ref{thm12} relies on the interior $W^{2,1+\varepsilon}$ regularity for viscosity solutions. We first briefly review the $W^{2,p}$ theory for several types of fully nonlinear equations. For the Monge-Ampère equation, the interior $W^{2,p}$ estimate was established by Caffarelli\,\cite{Caffarelli90W2p} for strictly convex solutions with continuous positive density. For densities with positive upper and lower bounds, De Philippis and Figalli\,\cite{DPF13W21} proved the $W_{loc}^{2,1}$ regularity for strictly convex solutions, and this result was improved to $W_{loc}^{2,1+\varepsilon}$ by De Philippis, Figalli and Savin\,\cite{DPFS13W21epsilon}, and independently, by Schmidt\,\cite{Schmidt13W21epsilon}. In dimension two, positive density assumption suffices to guarantee strict convexity, and the above estimates hold for general convex solutions. Later, Mooney\,\cite{Mooney15singular,Mooney15W21} removed the strict convexity condition for $W_{loc}^{2,1}$ regularity results in arbitrary dimensions. For the $k$-Hessian equation, Hu\,\cite{Hu26hessianW21} proved the $W_{loc}^{2,1}$ regularity for convex viscosity solutions with constant density. For the $2$-Hessian equation, Lyu\,\cite{Lyu26W2p} derived the interior $W^{2,1}$ (resp., $W^{2,p}$) estimates for $2$-subharmonic solutions with positive density bounded above and below (resp., with continuous positive density). In the next theorem, we establish the interior $W^{2,1+\varepsilon}$ regularity result for $2$-subharmonic functions.
\begin{thm}\label{thm13}
Suppose the same assumptions and notations as in Theorem \ref{thm12}. Then there exists a constant $\varepsilon=\varepsilon(n,\lambda,\Lambda)\in(0,1)$ such that $u\in W^{2,1+\varepsilon}_{loc}(\Omega)$.
In particular, if $B(x_0,4R)\subset\Omega$, then it holds
\begin{align*}
\int_{B(x_0,R)}|D^2u|^{1+\varepsilon}\leqs C(n,\lambda, \Lambda,\varepsilon,R,\Vert u\Vert_{L^\infty(B(x_0,4R))}).
\end{align*}
\end{thm}

\begin{remark}
After completing this work, we became aware of the very recent preprint \cite{Lyu26C1alpha}, in which the interior regularity results of Theorem \ref{thm12} and Theorem \ref{thm13} were established for smooth functions. The present work was completed independently of \cite{Lyu26C1alpha}. Our proofs make use of the interior gradient estimate established in Theorem \ref{thm11}.
\end{remark}

The rest of the paper is organized as follows. In Section 2, we establish the gradient integral estimates on $(k-1)$-convex domains (see \eqref{formu21} for definition). In particular, we prove that every continuous $k$-subharmonic functions with finite $k$-Hessian measure on $\Omega$ belongs to $W^{1,1}(\Omega)$. Section 3 is devoted to the proof of Theorem \ref{thm11}, namely the interior gradient estimates for $k$-Hessian equations. Section 4 and 5 are concerned with $2$-Hessian equations and complete the proofs of Theorem \ref{thm13} and Theorem \ref{thm12}, respectively.

\section{Gradient Integral Estimates}
In this section, we establish the gradient integral estimates for $k$-subharmonic functions on a bounded domain. Following \cite{CNS85Dirichlet}, we call a bounded $C^2$ domain $\Omega\subset\mathbb{R}^n$ to be $(k-1)$-convex, if $\partial\Omega$ is connected and
\begin{align}
\sigma_{k-1}(\kappa)\geqs c_0>0\quad\text{on }\partial\Omega
\label{formu21}
\end{align}
for some positive constant $c_0$, where $\kappa=(\kappa_1,\ldots,\kappa_{n-1})$ denote the principal curvatures of $\partial\Omega$ with respect to its inner normal. Let $\Omega$ be $(k-1)$-convex in this section.

\begin{lemma}[Comparison principle]\label{lem21}
Let $u,v\in\Phi^k(\Omega)\cap C^2(\overline{\Omega})$ satisfy $u\leqs v$ in $\Omega$ and $u=v$ on $\partial\Omega$. Let $w\in C^1(\overline{\Omega})$ satisfy $w\leqs u-v$ in $\Omega$ and $w=0$ on $\partial\Omega$. Then we have
\begin{align}
0\leqs\int_\Omega|\nabla u|-\int_\Omega|\nabla v|\leqs\int_\Omega|\nabla w|.
\label{formu22}
\end{align}
\end{lemma}
\begin{proof}
Denote $h=u-v\leqs0$, then $w\leqs h$. For $0\leqs t\leqs1$, define
\begin{align*}
U_t=(1-t)v+tu=v+th.
\end{align*}
By the convexity of $\Gamma_k$, $U_t\in\Phi^k(\Omega)\subset\Phi^2(\Omega)$. For $\varepsilon>0$, set
\begin{align*}
G_\varepsilon(p)=\sqrt{\varepsilon^2+|p|^2}\quad\text{and}\quad J_\varepsilon(t)=\int_\Omega G_\varepsilon(\nabla U_t).
\end{align*}
Denote $X_{t,\varepsilon}=DG_\varepsilon(\nabla U_t)$, then $|X_{t,\varepsilon}|\leqs1$ and
\begin{align*}
\textrm{div}X_{t,\varepsilon}=\textrm{div}\Big(\frac{\nabla U_t}{\sqrt{\varepsilon^2+|\nabla U_t|^2}}\Big)=\frac{\varepsilon^2\Delta U_t+\langle A\nabla U_t,\nabla U_t\rangle}{\big(\sqrt{\varepsilon^2+|\nabla U_t|^2}\big)^3},
\end{align*}
where $A=\Delta U_tI_n-D^2U_t$. Since $U_t\in\Phi^2(\Omega)$, we have (see \cite{Wangnote})
\begin{align*}
\Delta U_t\geqs0\quad\text{and}\quad DS_2[U_t]=\Delta U_tI_n-D^2U_t\geqs0.
\end{align*}
Therefore, we obtain $\text{div}X_\varepsilon\geqs0$. By integration by parts and $h=0$ on $\partial\Omega$,
\begin{align*}
J_\varepsilon'(t)=\int_\Omega\langle X_{t,\varepsilon},\nabla h\rangle=\int_\Omega(-h)\textrm{div}X_{t,\varepsilon}.
\end{align*}
Since $w\leqs h\leqs0$ and $\text{div}X_\varepsilon\geqs0$ in $\Omega$, by integration by parts again and $w=0$ on $\partial\Omega$,
\begin{align*}
0\leqs J_\varepsilon'(t)\leqs\int_\Omega(-w)\textrm{div}X_{t,\varepsilon}=\int_\Omega\langle X_{t,\varepsilon},\nabla w\rangle\leqs\int_\Omega|\nabla w|,
\end{align*}
where the last inequality follows from $|X_{t,\varepsilon}|\leqs1$. Integrating $J_\varepsilon'(t)$ from $0$ to $1$, we get
\begin{align*}
0\leqs\int_\Omega G_\varepsilon(\nabla u)-\int_\Omega G_\varepsilon(\nabla v)\leqs\int_\Omega|\nabla w|.
\end{align*}
We obtain the desired estimate by taking $\varepsilon\to0$.
\end{proof}
\begin{remark}\label{rmk21}
If $S_k[u]\geqs S_k[v]$ in $\Omega$ and $u=v$ on $\partial\Omega$, then by the comparison principle, $u\leqs v$ in $\Omega$ and thus \eqref{formu22} holds under this assumption.
\end{remark}

The next lemma concerns $k$-subharmonic functions vanishing on the boundary $\partial\Omega$, where the geometric condition on $\Omega$ comes into play.
\begin{lemma}\label{lem22}
Let $u\in\Phi^k(\Omega)\cap C^2(\overline{\Omega})$ satisfy $u=0$ on $\partial\Omega$. Then
\begin{align}
\int_\Omega|\nabla u|\leqs C\Big(\int_\Omega S_k[u]\Big)^{1/k},
\label{formu2new1}
\end{align}
where the constant $C$ depends on $n,k$ and $\Omega$.

In particular, if $\Omega$ is a ball of radius $R$ in $\mathbb{R}^n$, then $C=C(n,k)R^{\frac{n(k-1)+k}{k}}$.
\end{lemma}
\begin{proof}
Denote by $\omega_n$ the volume of the unit ball in $\mathbb{R}^n$.
By Talenti's gradient rearrangement inequality \cite{Talenti76},
\begin{align}
\Vert\nabla u\Vert_{L^q(\Omega)}\leqs\Vert\nabla v\Vert_{L^q(\Omega^*)}\quad\text{for } q\in(0,2],
\label{formu23}
\end{align}
where $\Omega^*=B_{R_\Omega}$ with $R_\Omega=(\frac{|\Omega|}{\omega_n})^{1/n}$, and $v$ is the radial solution of
\begin{equation*}
\left\{
\begin{array}{cc}
-\Delta v=f^\# &\text{in }\Omega^*,\\
v=0 &\text{on }\partial\Omega^*.
\end{array}
\right.
\end{equation*}
Here, $f^\#$ denotes the spherically symmetric decreasing rearrangement of $|\Delta u|=\Delta u$. Since $v$ is radial, we have
\begin{align*}
n\omega_nr^{n-1}(-v'(r))=\int_{B_r}f^\#.
\end{align*}
Hence,
\begin{align*}
\int_{\Omega^*}|\nabla v|=\int_0^{R_\Omega}dr\int_{B_r}f^\#=\int_{\Omega^*}(R_\Omega-|x|)f^\#\leqs R_\Omega\int_{\Omega}\Delta u.
\end{align*}
Applying this to \eqref{formu23} with $q=1$, we obtain
\begin{align}
\int_\Omega|\nabla u|\leqs\int_{\Omega^*}|\nabla v|\leqs R_\Omega\int_\Omega\Delta u.
\label{formu24}
\end{align}

By the variational structure of $\sigma_k$ (see \cite{Wangnote}) and $u=0$ on $\partial\Omega$,
\begin{align*}
k\int_\Omega\sigma_k(D^2u)=\int_\Omega\partial_i\big(\sigma_k^{ij}(D^2u)u_j\big)=\int_{\partial\Omega}\sigma_k^{ij}(D^2u)\nu_iu_j,
\end{align*}
where $\nu$ denotes the outer normal of $\partial\Omega$. For any $x_0\in\partial\Omega$, by a rotation of the coordinates, we assume $\nu=-e_n$ and locally $\partial\Omega$ is given by $x_n=\rho(x')$ where $x'=(x_1,\ldots,x_{n-1})$. Then $D\rho(x_0')=0$ and differentiating the boundary condition $u=0$ twice yields
\begin{align*}
u_{ij}(x_0)+u_n(x_0)\rho_{ij}(x_0)=0.
\end{align*}
By the geometric assumption \eqref{formu21}, we obtain
\begin{align*}
\sigma_k^{nn}(D^2u)=\sigma_{k-1}([D^2u]')=u_\nu^{k-1}\sigma_{k-1}(\kappa)\geqs c_0u_\nu^{k-1},
\end{align*}
where $[D^2u]'=(u_{ij})_{1\leqs i,j\leqs n-1}$. Therefore, by H\"older's inequality,
\begin{align}
k\int_\Omega\sigma_k(D^2u)\geqs c_0\int_{\partial\Omega} u_\nu^k\geqs c_0|\partial\Omega|^{1-k}\Big(\int_{\partial\Omega} u_\nu\Big)^k.
\label{formu25}
\end{align}
Combining \eqref{formu24}$\sim$\eqref{formu25} and utilizing the divergence theorem, we have
\begin{align}
\int_\Omega|\nabla u|\leqs R_\Omega\big(\frac{k}{c_0} \big)^{\frac{1}{k}}|\partial\Omega|^{1-\frac{1}{k}} \Big(\int_\Omega S_k[u]\Big)^{1/k},
\label{formu2new2}
\end{align}
and this proves \eqref{formu2new1}.

If $\Omega=B_R$, then $R_\Omega=R$ and $|\partial B_R|=n\omega_nR^{n-1}$. The second fundamental form of $\partial B_R$ is given by $R^{-1}I$. Consequently, we can take $c_0=\sigma_{k-1}(\kappa)=\binom{n-1}{k-1}R^{1-k}$. Substituting these expressions into \eqref{formu2new2} yields
\begin{align*}
\int_{B_R}|\nabla u|\leqs\bigg(\frac{k(n\omega_n)^{k-1}}{\binom{n-1}{k-1}}\bigg)^{\frac{1}{k}}R^{\frac{n(k-1)+k}{k}}\Big(\int_{B_R} S_k[u]\Big)^{1/k}.
\end{align*}
This completes the proof of Lemma \ref{lem22}.
\end{proof}

The main result in this section is as follows.
\begin{proposition}\label{pro21}
Let $u\in\Phi^k(\Omega)\cap C^0(\overline{\Omega})$ satisfy $\mu_k(u)=\mu$ in $\Omega$ and $\mu$ be a finite Borel measure on $\Omega$. Then $u\in W^{1,1}(\Omega)$ and it satisfies
\begin{align*}
\Vert\nabla u\Vert_{L^1(\Omega)}\leqs C\mu(\Omega)^{1/k}+ C_1\operatorname*{osc}_{\partial\Omega}u,
\end{align*}
where $C$ is the same constant as in \eqref{formu2new1}, and $C_1=|\partial\Omega|$.
\end{proposition}
\begin{proof}
Denote $g=u\big|_{\partial\Omega}$. Let $g_j\in C^\infty(\partial\Omega)$ be a sequence converging to $g$ in $C(\partial\Omega)$. We first assume that $\mu$ is compactly supported in $\Omega'\Subset\Omega$. Let $f_j$ be a sequence of nonnegative functions in $C^\infty_c(\Omega')$ converging weakly as measures to $\mu$. From \cite{CNS85Dirichlet}, there exists a smooth solution $u_j$ of
\begin{equation*}
\left\{
\begin{array}{cc}
S_k[u_j]=f_j+2\delta_j &\text{in }\Omega,\\
u_j=g_j &\text{on }\partial\Omega,
\end{array}
\right.
\end{equation*}
where $\delta_j>0$ tends to zero as $j\to\infty$. Let $v_j$, $w_j$ be the solution of
\begin{equation*}
\left\{
\begin{array}{cc}
S_k[v_j]=\delta_j  & \text{in }\Omega \\
v_j=g_j  & \text{on }\partial\Omega
\end{array}
\right.\quad\text{and}\quad
\left\{
\begin{array}{cc}
S_k[w_j]=f_j+\delta_j  & \text{in }\Omega \\
w_j=0  & \text{on }\partial\Omega
\end{array}
\right..
\end{equation*}
By the concavity of $\sigma_k^{1/k}$, $S_k[v_j+w_j]\geqs S_k[v_j]+S_k[w_j]=S_k[u_j]$. Since $u_j=v_j=v_j+w_j$ on $\partial\Omega$, by the comparison principle, we have $v_j+w_j\leqs u_j\leqs v_j$ in $\Omega$. By Lemma \ref{lem21},
\begin{align}
0\leqs\int_\Omega|\nabla u_j|-\int_\Omega|\nabla v_j|\leqs\int_\Omega|\nabla w_j|.
\label{formu26}
\end{align}

Let $m_j=\inf_\Omega v_j$. By the $L^\infty$-estimate for solutions, see \cite[Theorem 5.5]{Wangnote},
\begin{align}
|m_j|\leqs|\inf_{\partial\Omega}g_j|+C\delta_j^{1/k}.
\label{formu27}
\end{align}
For $\varepsilon>0$, denote $X_{j,\varepsilon}=\frac{\nabla v_j}{\sqrt{\varepsilon^2+|\nabla v_j|^2}}$. As in the proof of Lemma \ref{lem21}, we have $|X_{j,\varepsilon}|\leqs1$ and $\mathrm{div}X_{j,\varepsilon}\geqs0$. Therefore,
\begin{align*}
\int_\Omega\mathrm{div}[X_{j,\varepsilon}(v_j-m_j)]= \int_\Omega\frac{|\nabla v_j|^2}{\sqrt{\varepsilon^2+ |\nabla v_j|^2}}+\int_\Omega\mathrm{div}X_{j,\varepsilon}(v_j-m_j)\geqs\int_\Omega\frac{|\nabla v_j|^2}{\sqrt{\varepsilon^2+ |\nabla v_j|^2}}.
\end{align*}
Applying the divergence theorem, we obtain
\begin{align*}
\int_\Omega\frac{|\nabla v_j|^2}{\sqrt{\varepsilon^2+ |\nabla v_j|^2}}\leqs\int_{\partial\Omega}\langle X_{j,\varepsilon},\nu\rangle(v_j-m_j)\leqs \int_{\partial\Omega}g_j-m_j.
\end{align*}
By taking $\varepsilon\to0$, we get
\begin{align}
\int_\Omega|\nabla v_j|\leqs\int_{\partial\Omega}g_j-m_j.
\label{formu28}
\end{align}
On the other side, by Lemma \ref{lem22},
\begin{align}
\int_\Omega|\nabla w_j|\leqs C\Big(\int_\Omega S_k[w_j] \Big)^{1/k}=C\Big(\int_\Omega f_j+\delta_j\Big)^{1/k}.
\label{formu29}
\end{align}
Combining \eqref{formu26}$\sim$\eqref{formu29}, we obtain
\begin{align}
\int_\Omega|\nabla u_j|\leqs C\Big(\int_\Omega f_j \Big)^{1/k}+\int_{\partial\Omega}(g_j-\inf_{\partial\Omega} g_j)+\widetilde C\delta_j^{1/k}.
\label{formu210}
\end{align}

By the construction of $u_j$, we have $\{u_j\}$ converges to $u$ in $L^1_{loc}(\Omega)\cap C^0(\Omega-\overline{\Omega'})$. We refer the reader to \cite[Section 6]{TW97measureII} for details. Thus, $u_j$ converges to $u$ almost everywhere in $\Omega$. Since $\Phi^k(\Omega)\subset W^{1,q}_{loc}(\Omega)$ for $q<\frac{nk}{n-k}$, $Du_j$ converges weakly in $L^q_{loc}(\Omega)$ to $Du$. By the weak lower semi-continuity, passing to the limit in \eqref{formu210} yields, for any $K\Subset\Omega$,
\begin{align*}
\int_K|\nabla u|\leqs \liminf_{j\to\infty}\int_K|\nabla u_j|\leqs C\mu(\Omega)^{1/k}+\int_{\partial\Omega}(g-\inf_{\partial\Omega} g).
\end{align*}
Hence,
\begin{align*}
\int_\Omega|\nabla u|\leqs C\mu(\Omega)^{1/k}+ |\partial\Omega|\operatorname*{osc}_{\partial\Omega}g.
\end{align*}

Now we drop the assumption that $\mu$ is compactly supported in $\Omega$. Let $0\leqs\chi_i\leqs1$ be such that $\chi_i(x)=1$ for $\mathrm{dist}(x,\partial\Omega)\geqs\frac{1}{i}$ and $\chi_i(x)=0$ for $\mathrm{dist}(x,\partial\Omega)\leqs\frac{1}{2i}$. Then $u$ is an admissible viscosity subsolution to the following Dirichlet problem
\begin{equation}
\left\{
\begin{array}{cc}
\mu_k[u]=\chi_i\mu &\text{in }\Omega,\\
u=g &\text{on }\partial\Omega.
\end{array}
\right.
\label{formu211}
\end{equation}
Thus, there exists a viscosity solution $u_i\in\Phi^k(\Omega)$ to \eqref{formu211} and $u_i\geqs u$ in $\Omega$. By a similar argument as above, $Du_i$ converges weakly in $L^q_{loc}(\Omega)$ to $Du$ and for any $K\Subset\Omega$,
\begin{align*}
\int_K|\nabla u|\leqs \liminf_{i\to\infty}\int_K|\nabla u_i|\leqs C\mu(\Omega)^{1/k}+ |\partial\Omega|\operatorname*{osc}_{\partial\Omega}g.
\end{align*}
By the arbitrariness of $K$, we finally conclude the proof.
\end{proof}

\begin{remark}
Proposition \ref{pro21} shows that any continuous k-subharmonic function which has finite $k$-Hessian measure belongs to $W^{1,1}(\Omega)$. It was recently proven in \cite{Gehring26nonsmooth} that the following Dirichlet problem is solvable in $\Phi^k(\Omega)\cap C^0(\overline{\Omega})$ for $k>n/2$,
\begin{equation*}
\left\{
\begin{array}{cc}
\mu_k[u]=\mu &\text{in }\Omega,\\
u=g &\text{on }\partial\Omega,
\end{array}
\right.
\end{equation*}
provided that $\mu(\Omega)<\infty$ and $g\in C(\partial\Omega)$. Under this case, the solution belongs to $W^{1,1}(\Omega)$.
\end{remark}

\begin{corollary}\label{cor21}
Under the hypotheses of Proposition \ref{pro21}, let $\Omega=B_R$ and $u\in C^{0,1}(\partial B_R)$. Then $u$ satisfies
\begin{align}
\frac{1}{R^n}\int_{B_R}|\nabla u|\leqs C(n,k)\Big(\frac{\mu(B_R)}{R^{n-k}}\Big)^{1/k}+C_n\max_{\partial B_R}|\nabla_\tau u|,
\label{formu214}
\end{align}
where $\tau$ ranges over all unit tangent vectors to $\partial B_R$.
\end{corollary}
\begin{proof}
By Lemma \ref{lem22} and Proposition \ref{pro21},
\begin{align}
\int_{B_R}|\nabla u|\leqs C(n,k)R^{\frac{n(k-1)+k}{k}}\mu(B_R)^{1/k}+n\omega_nR^{n-1}\operatorname*{osc}_{\partial B_R}u.
\label{formu215}
\end{align}
Let $M=\max_{\partial B_R}|\nabla_\tau u|$. Then for any $x,y\in\partial B_R$,
\begin{align*}
|u(x)-u(y)|\leqs\int_\gamma|D_\tau u|ds\leqs c_nRM,
\end{align*}
where $\gamma\subset\partial B_R$ is a $C^1$ curve connecting $x$ and $y$. Taking the supremum over $x,y\in\partial B_R$, we have $\operatorname*{osc}_{\partial B_R}u\leqs c_nRM$. Then, dividing \eqref{formu215} by $R^n$ yields
\begin{align*}
\frac{1}{R^n}\int_{B_R}|\nabla u|\leqs C(n,k)\Big(\frac{\mu(B_R)}{R^{n-k}}\Big)^{1/k}+C_nM.
\end{align*}
\end{proof}
\begin{remark}\label{rmk23}
It is obvious that $|\nabla_\tau u|\leqs|\nabla u|$ on $\partial B_R$. In fact, estimate \eqref{formu214} shows that the average of $|\nabla u|$ over a ball is controlled by the boundary values of $|\nabla u|$ and a potential term involving the $k$-Hessian measure. This estimate will be used in the next section.
\end{remark}

\section{Interior Gradient Estimates}
We now establish the interior gradient estimates for $k$-subharmonic function by means of the Wolff potential of its $k$-Hessian measure. The proof is based on the method developed by Labutin\,\cite{Labutin02potential}.

\subsection{Lipschitz Stability Lemma}
We start with an algebraic lemma for $\sigma_k$. For matrices $A=(a_{ij})$, we write $A\in\Gamma_k$ if $\lambda(A)\in\Gamma_k$ and denote $\sigma_k(A)=\sigma_k(\lambda(A))$, $T(A)=D_A\sigma_k(A)$.
\begin{lemma}\label{lem31}
Let $A\in\Gamma_k$. If $e\in\mathbb{S}^{n-1}$ satisfy $A[e,e]\leqs0$, then
\begin{align}
T(A)[e,e]\geqs\frac{\operatorname{tr}T(A)}{n-k+1}.
\label{formu31}
\end{align}
\end{lemma}
\begin{proof}
Choose an orthonormal basis $\{e_i\}_{i=1}^n$ such that $e_1=e$. $A$ can be written as
\begin{equation*}
A=
\begin{pmatrix}
a & b^T \\
b & B
\end{pmatrix},\qquad
a=A[e,e]\leqs0.
\end{equation*}
Then $B\in\Gamma_{k-1}$ and $T(A)[e,e]=\sigma_{k-1}(B)>0$. By direct calculation,
\begin{align*}
\sigma_{k-1}(A)=\sigma_{k-1}(B)+a\sigma_{k-2}(B)-b^T D\sigma_{k-2}(B)b\leqs\sigma_{k-1}(B),
\end{align*}
since $D\sigma_{k-2}(B)$ is positive semidefinite. Applying
\begin{align*}
\sigma_{k-1}(A)=\frac{\operatorname{tr}T(A)}{(n-k+1)} \quad\text{and}\quad\sigma_{k-1}(B)=T(A)[e,e],
\end{align*}
we obtain the desired inequality.
\end{proof}

Next, we prove the Lipschitz stability lemma for latter applications. Our proof adapts the two-point maximum principle  argument by Ishii and Lions\,\cite{IL90viscosity}. The key tool is the Hessian comparison at the maximum value, which is obtained by testing the second variation along the reflected pairs $(\xi,R\xi)$, see \eqref{formu34}.
\begin{lemma}\label{lem32}
Let $u\in\Phi^k(B_r)$ satisfy $S_k[u]=c\geqs0$. Then for any Lipschitz function $v$,
\begin{align}
\operatorname{Lip}_{B_{r/2}}u\leqs\operatorname{Lip}_{B_r}v+\frac{C(n,k)}{r}\Vert u-v\Vert_{L^\infty(B_r)}.
\label{formu32}
\end{align}
\end{lemma}
\begin{proof}
We first assume that $r=1$. Let $\tau>0$ to be determined later. Set
\begin{align*}
\phi(t)=\frac{8}{\tau}(t-\frac{t^2}{2\tau})\quad\text{for }t\in[0,\tau].
\end{align*}
Then $\phi(\tau)=4$, $0\leqs\phi'\leqs\frac{8}{\tau}$ and $\phi''=-\frac{8}{\tau^2}$.

Fix any $x_0\in B_{1/2}$. Denote
\begin{align*}
M=\operatorname{Lip}_{B_1}v,\quad L=64,\quad s=|x-y|,\quad z=\frac{x+y}{2}\quad\text{and}\quad\varepsilon=\Vert u-v\Vert_{L^\infty(B_1)}.
\end{align*}
Consider the auxiliary two-point function $Z(x,y)$
\begin{align*}
Z(x,y)=u(x)-u(y)-Ms-\varepsilon(\phi(s)+L|z-x_0|^2),
\end{align*}
defined on $z\in\overline{B_{1/4}(x_0)}$ and $s\in[0,\tau]$. Let $\tau>0$ be small enough so that $x,y\in B_{7/8}$.

Since $M=\operatorname{Lip}_{B_1}v$, we have
\begin{align*}
u(x)-u(y)-Ms\leqs2\Vert u-v\Vert_{L^\infty(B_1)}+v(x)-v(y)-M|x-y|\leqs2\varepsilon.
\end{align*}
Moreover, $\phi(\tau)=4$ and $L|z-x_0|^2=4$ on $\partial B_{1/4}(x_0)$. Therefore, $Z(x,y)\leqs-2\varepsilon$ on $s=\tau$ or $z\in\partial B_{1/4}(x_0)$. When $s=0$, $Z(x,y)=-\varepsilon L|z-x_0|^2\leqs0$. It follows that any positive maximum value of $Z(x,y)$, if it exists, must be attained in the interior of the domain. Suppose that $Z(x_0,y_0)=\max Z(x,y)>0$ with $x_0\ne y_0$. Denote
\begin{align*}
s_0=|x_0-y_0|,\quad z_0=\frac{x_0+y_0}{2},\quad X=D^2u(x_0)\quad\text{and}\quad Y=D^2u(y_0).
\end{align*}

We next establish a relation between $X$ and $Y$. Set
\begin{align*}
e=\frac{x_0-y_0}{s_0},\quad P=I-e\otimes e\quad\text{and} \quad R=I-2e\otimes e.
\end{align*}
Then $P$ is the orthogonal projection onto $e^{\perp}$ and $R$ is the reflection across $e^{\perp}$.
\begin{itemize}[leftmargin=11pt]
\item Along $\lambda\mapsto(x_0+\lambda e,y_0)$,
\begin{align*}
s'=1,\quad s''=0,\quad z'=\frac{e}{2},\quad z''=0.
\end{align*}
At the maximum point $(x_0,y_0)$,
\begin{align*}
0\geqs\frac{d^2}{d\lambda^2}\bigg|_{\lambda=0}Z(x,y)=X[e,e]-\varepsilon\Big(\phi''+\frac{L}{2}\Big).
\end{align*}
Hence,
\begin{align}
X[e,e]\leqs\varepsilon\Big(-\frac{8}{\tau^2}+\frac{L}{2}\Big).
\label{formu33}
\end{align}
\item Along $\lambda\mapsto(x_0+\lambda\xi,y_0+\lambda R\xi)$, where $\xi=\alpha e+\eta$ and $\eta\perp e$,
\begin{align*}
s'=2\alpha,\quad s''=0,\quad z'=\eta,\quad z''=0.
\end{align*}
Similarly, at the maximum point $(x_0,y_0)$,
\begin{align*}
0\geqs\frac{d^2}{d\lambda^2}\bigg|_{\lambda=0}Z(x,y)= X[\xi,\xi]-Y[R\xi,R\xi]-\varepsilon\Big(\phi''(2\alpha)^2+2L|\eta|^2\Big).
\end{align*}
Since $\langle P\xi,\xi\rangle=\langle\eta,\xi\rangle=|\eta|^2$, we obtain
\begin{align}
X-R^TYR\leqs\varepsilon\Big(-\frac{32}{\tau^2}e\otimes e+2LP\Big).
\label{formu34}
\end{align}
\end{itemize}

Set $\beta=n-k+1$ and $\delta=\frac{4}{\beta\tau^2}$. Consider the one-parameter family of matrices
\begin{align*}
X_\gamma=X+\gamma\varepsilon(\beta e\otimes e-I), \quad\gamma\in[0,\delta].
\end{align*}
By \eqref{formu33},
\begin{align*}
X_\gamma[e,e]=X[e,e]+\gamma\varepsilon(\beta-1)\leqs\varepsilon\Big(-\frac{4}{\tau^2}+\frac{L}{2}\Big)\leqs0,
\end{align*}
when $\tau^2\leqs\frac{8}{L}$. According to Lemma \ref{lem31}, we have
\begin{align*}
\frac{d}{d\gamma}\sigma_k(X_\gamma)=\Big\langle T(X_\gamma),\frac{dX_\gamma}{d\gamma}\Big\rangle =\varepsilon\Big(\beta T(X_\gamma)[e,e]-\operatorname{tr}T(X_\gamma)\Big)\geqs0.
\end{align*}
Since $X\in\Gamma_k$, it follows that $X_\gamma\in \Gamma_k$ for all $\gamma\in[0,\delta]$ and $\sigma_k(X)\leqs\sigma_k(X_\delta)$. On the other hand, by virtue of \eqref{formu34},
\begin{align*}
X_\delta=X+\varepsilon\Big(\frac{4}{\tau^2}e\otimes e-\frac{4}{\beta\tau^2}I\Big)\leqs X+\varepsilon\Big(\frac{32}{\tau^2}e\otimes e-4LI\Big)\leqs R^TYR-2\varepsilon LI,
\end{align*}
provided that $\tau^2\leqs\frac{1}{\beta L}$. Thus,
\begin{align*}
\sigma_k(X)\leqs\sigma_k(X_\delta)<\sigma_k(R^TYR) =\sigma_k(Y),
\end{align*}
which contradicts $\sigma_k(D^2u)=c$. Therefore, $Z(x,y)\leqs0$ holds.

For any $e\in\mathbb{S}^{n-1}$, 
take $x=x_0+he$ and $y=x_0-he$. Then, $z=x_0$ and
\begin{align*}
u(x_0+he)-u(x_0-he)\leqs 2Mh+\varepsilon\phi(2h).
\end{align*}
Dividing both sides by $2h$ and letting $h\to0$, we obtain
\begin{align*}
Du(x_0)\cdot e\leqs M+\frac{8}{\tau}\varepsilon.
\end{align*}
By the arbitrariness of $x_0\in B_{1/2}$ and $e\in\mathbb{S}^{n-1}$, we get
\begin{align*}
\operatorname{Lip}_{B_{1/2}}u\leqs\operatorname{Lip}_{B_{1}}v+C(n,k)\Vert u-v\Vert_{L^\infty(B_1)}.
\end{align*}
The desired estimate \eqref{formu32} then follows by a rescaling.
\end{proof}

\subsection{Interior Gradient Bounds}
We shall employ several estimates in \cite{TW97measureII} and \cite{Labutin02potential}. Following \cite{GT83elliptic}, we define the weighted norms on $C^{0,1}(\Omega)$, for $\sigma\geqs0$,
\begin{align*}
|u|_{0,1;\Omega}^{(\sigma)}=\sup_{x\in\Omega}d_x^\sigma|u(x)|+\sup_{x,y\in\Omega,x\neq y}d_{x,y}^{\sigma+1}\frac{|u(x)-u(y)|}{|x-y|},
\end{align*}
where $d_x=\operatorname{dist}(x,\partial\Omega)$ and $d_{x,y}=\min(d_x,d_y)$.

Let $u\in\Phi^k(B_R)$ satisfy $S_k[u]\equiv c_0\geqs0$ in $B_R$. Then, by \cite[Corollary 3.4]{TW97measureII},
\begin{align*}
|u|_{0,1;B_R}^{(n)}\leqs C(n)\int_{B_R}|u|.
\end{align*}
Accordingly, for $0<r<R$,
\begin{align}
|u(0)|\leqs \frac{C}{r^n}\int_{B_r}|u|
\label{formu35}
\end{align}
and
\begin{align}
|Du(0)|\leqs\frac{C}{r^{n+1}}\int_{B_r}|u|.
\label{formu36}
\end{align}
Applying \eqref{formu36} to $u-u_r$ with $u_r=\fint_{B_r} u$, we obtain
\begin{align}
|Du(0)|\leqs\frac{C}{r^{n+1}}\int_{B_r}|u-u_r| \leqs\frac{C}{r^n}\int_{B_r}|\nabla u|,
\label{formu37}
\end{align}
where the last inequality follows by the $L^1$ Poincar\'e inequality.

Our estimate involves the classical Wolff potential defined by
\begin{align*}
\mathbf{W}_k^\mu(x,r)=\int_0^r\Big(\frac{\mu(B(x,t))}{t^{n-k}}\Big)^{1/k}\frac{dt}{t}.
\end{align*}
For a nonnegative Radon measure $\mu$, a dyadic decomposition of the interval yields the two-sided estimate,
\begin{align}
\frac{1}{C}\sum_{j=1}^\infty\Big(\frac{\mu(B(x,R2^{-j}))}{(R2^{-j})^{n-k}}\Big)^{1/k}\leqs\mathbf{W}_k^\mu(x,R)\leqs C\sum_{j=0}^\infty\Big(\frac{\mu(B(x,R2^{-j}))}{(R2^{-j})^{n-k}}\Big)^{1/k}.
\label{formu38}
\end{align}

The following estimate establishes a relation between the $L^1$-norm of a $k$-subharmonic function and its Hessian measure. Indeed, if $v\in C^\infty(\overline{B}_r)$ satisfies $v=0$ on $\partial B_r$, then
\begin{align}
\frac{1}{r^n}\Big|\int_{B_r}v\Big|\leqs C(n,k)\Big(\frac{1}{r^{n-2k}}\int_{B_r}S_k[v]\Big)^{1/k}.
\label{formu39}
\end{align}
The proof relies on the variational structure of $S_k$ and the Hessian Sobolev inequalities\,\cite{Wang94class}; see \cite[(2.22)]{Labutin02potential} for details. This estimate remains valid for the weak solutions. Indeed, let $v\in\Phi^k(B_r)\cap C^0(B_r-\overline{B_\rho})$ solve $\mu_k(v)=\mu$ in $B_r$ and $v=0$ on $\partial B_r$, where $0<\rho<r$ and $\mu$ is a finite Radon measure on $B_r$. We first assume that $\mu$ is compactly supported in $B_r$. As discussed in Proposition \ref{pro21}, there exists a smooth solution $v_j$ to the Dirichlet problem
\begin{equation*}
\left\{
\begin{array}{cc}
S_k[v_j]=f_j+\delta_j  & \text{in }B_r, \\
v_j=0  & \text{on }\partial B_r,
\end{array}
\right.
\end{equation*}
where $f_j\in C_c^\infty(B_r)$ converges weakly in measure to $\mu$, and $\delta_j>0$ tends to zero as $j\to\infty$. Similarly, $v_j$ converges to $v$ in $L^1_{loc}(B_r)\cap C^0(B_r-\overline{B_\rho})$. For any $K\Subset B_r$, we have
\begin{align*}
\frac{1}{r^n}\Big|\int_{K}v\Big|&=\lim_{j\to\infty}\frac{1}{r^n}\Big|\int_{K}v_j\Big|\\
&\leqs\lim_{j\to\infty}C(n,k)\Big(\frac{1}{r^{n-2k}}\int_{B_r}S_k[v_j]\Big)^{1/k}\\
&=C(n,k)\Big(\frac{\mu(B_r)}{r^{n-2k}}\Big)^{1/k},
\end{align*}
where the inequality follows from \eqref{formu39}. This implies
\begin{align}
\frac{1}{r^n}\Big|\int_{B_r}v\Big|\leqs C(n,k)\Big(\frac{\mu(B_r)}{r^{n-2k}}\Big)^{1/k}.
\label{formu310}
\end{align}
By a truncation argument, we can derive \eqref{formu310} for general finite measure $\mu$.
\\[-0.5em]

Now, we are ready to prove Theorem \ref{thm11}.
~\\[0.3em]
\textbf{Proof of Theorem \ref{thm11}.}

We divide the proof into two steps. The first step is to establish the estimate for the average gradient of smooth functions, involving only a finite number of potential terms in $\eqref{formu38}$. The second step is to extend the estimate to general $k$-subharmonic functions by standard mollification.
\\[0.3em]
\textbf{Step 1.} By a translation of coordinates, assume that $x_0=0$. Let $u\in\Phi^k(B_{3R})\cap C^\infty(B_{3R})$. In this step, we prove the following estimate for $0<r<R$,
\begin{align}
\fint_{B_r}|\nabla u|\leqs C\fint_{B_R}|\nabla u|+C\sum_{i=0}^{N(r)}\Big(\frac{S_k[u](\overline{B}_{R2^{-i}})}{(R2^{-i})^{n-k}}\Big)^{1/k},
\label{formu311}
\end{align}
where $N(r)\in\mathbb{N}$ is the minimal integer such that $R2^{-N(r)}<r$, and $C=C(n,k)$. Without loss of generality, assume that $N(r)$ is even, $N(r)=2p$ for some $p\geqs3$.

Denote $r_l=R2^{-2l}$ for $l=0,\ldots,p-1$. Then, $r_{l-1}=4r_l$. Define $E_l=B_{r_l}-B_{r_l/2}$ and let $E=\cup_{l=0}^{p-1}E_l$. In the following, we introduce a specific modification $U$ of $u$. Let
\begin{align*}
U=u\quad\text{in }B_{2R}-E,
\end{align*}
and in the shell $E_l$, let $U$ be the $k$-convex solution to the Dirichlet problem
\begin{equation}
\left\{
\begin{array}{cc}
S_k[U]=0  & \text{in }E_l, \\
U=u  & \text{on }\partial E_l.
\end{array}
\right.
\label{formu312}
\end{equation}
Since $u\in\Phi^k(E_l)\cap C^\infty(\overline{E_l})$ is by itself a subsolution, the viscosity solution of \eqref{formu312} exists and belongs to $\Phi^k(E_l)\cap C^0(\overline{E_l})$, which follows from the standard Perron process (see \cite{Trudinger90Perron} and \cite{TW97measureII}). By the comparison principle, $U\geqs u$ in $B_{2R}$. Moreover, $U\in\Phi^k(B_{2R})\cap C^0(\overline{B}_{2R})$ by the arguments from viscosity theory (or by the definition of $k$-subharmonic functions).

Set $\Omega_l=B_{3r_l/4}$ and $P_l=\partial\Omega_l$ for $l=0,\ldots,p-1$. Obviously,
\begin{align}
P_l\subset E_l\quad\text{and}\quad\operatorname{dist}(P_l,\partial E_l)=\frac{r_l}{4}.
\label{formu313}
\end{align}
Let $h_l\in\Phi^k(\Omega_l)\cap C^0(\overline{\Omega_l})$ be the solution to the homogeneous Dirichlet problem
\begin{equation}
\left\{
\begin{array}{cc}
S_k[h_l]=0  & \text{in }\Omega_l, \\
h_l=U  & \text{on }\partial\Omega_l.
\end{array}
\right.
\label{formu314}
\end{equation}
\\
\textbf{(1). Estimate on $h_0$}

Let $H_0=h_0$ in $B_{3R/4}$ and $H_0=U$ in $B_{2R}-B_{3R/4}$. Then, $H_0\in\Phi^k(B_{2R})\cap C^0(\overline{B}_{2R})$. By the comparison principle, $H_0\geqs U\geqs u$. Let $\varphi_\varepsilon$ be a standard mollifier supported in $B_\varepsilon$. Denote $H_\varepsilon=H_0*\varphi_\varepsilon$ and $u_\varepsilon=u*\varphi_\varepsilon$. Since $H_0=u$ in $B_{2R}-B_R$, it follows that $H_\varepsilon=u_\varepsilon$ on $\partial B_{R+\varepsilon}$ and $H_\varepsilon\geqs u_\varepsilon$ in $B_{R+\varepsilon}$. By Lemma \ref{lem21},
\begin{align*}
\int_{B_{R+\varepsilon}}|\nabla H_\varepsilon|\leqs \int_{B_{R+\varepsilon}}|\nabla u_\varepsilon|.
\end{align*}
Since $\Phi^k(B_{2R})\subset W^{1,q}_{loc}(B_{2R})$ for $q<\frac{nk}{n-k}$, then $DH_\varepsilon$ converges weakly in $L^q(B_R)$ to $DH_0$. Using the weak lower semi-continuity, we have (note that $u\in C^\infty$)
\begin{align*}
\int_{B_R}|\nabla H_0|\leqs\liminf_{\varepsilon\to0} \int_{B_R}|\nabla H_\varepsilon|\leqs\lim_{\varepsilon\to0} \int_{B_{R+\varepsilon}}|\nabla u_\varepsilon| =\int_{B_R}|\nabla u|.
\end{align*}
In particular,
\begin{align*}
\int_{B_{3R/4}}|\nabla h_0|\leqs\int_{B_R}|\nabla u|.
\end{align*}
Since $S_k[h_0]=0$, we obtain the following Lipschitz estimate by \eqref{formu37},
\begin{align}
\operatorname*{Lip}_{B_{R/4}}h_0\leqs\frac{C}{R^n}\int_{B_{3R/4}}|\nabla h_0|\leqs\frac{C}{R^n}\int_{B_R}|\nabla u|.
\label{formu315}
\end{align}
\\
\textbf{(2). Estimate on $h_l$}

By the comparison principle, $h_l\geqs U$ in $B_{3r_l/4}$. Restricted in $B_{3r_l/4}$, $S_k[U]$ is supported on $\overline{B}_{r_l/2}\Subset B_{3r_l/4}$. By approximation, there exists a solution $w_l$ to the Dirichlet problem
\begin{equation}
\left\{
\begin{array}{cc}
S_k[w_l]=S_k[U]  & \text{in }B_{3r_l/4}, \\
w_l=0  & \text{on }\partial B_{3r_l/4}.
\end{array}
\right.
\label{formu3new1}
\end{equation}
By the concavity of $\sigma_k^{1/k}$, $S_k[w_l+h_l]\geqs S_k[w_l]+S_k[h_l]=S_k[U]$. Applying the comparison principle again, we obtain $w_l+h_l\leqs U$ in $B_{3r_l/4}$.

Recall that $S_k[U]=0$ in $E_{l+1}=B_{r_{l+1}}-B_{r_{l+1}/2}$. By \eqref{formu35} and \eqref{formu313},
\begin{align*}
|w_l(x)|\leqs\frac{C}{r_{l+1}^n}\int_{B_{r_{l+1}}}|w_l| \leqs\frac{C'}{r_l^n}\Big|\int_{B_{3r_l/4}}w_l\Big|
\end{align*}
holds for all $x\in P_{l+1}$. Together with \eqref{formu310}, it yields
\begin{align}
\Vert w_l\Vert_{L^\infty(P_{l+1})}\leqs C\Big(\frac{S_k[U](B_{3r_l/4})}{r_l^{n-2k}}\Big)^{1/k}.
\label{formu316}
\end{align}
It was proved in \cite{TW97measureI} and \cite{TW97measureIII} that for bounded functions $g,h\in\Phi^k(\Omega)$,
\begin{align}
g=h\text{ on }\partial\Omega,\quad g\geqs h\text{ in }\Omega\quad\Longrightarrow\quad S_k[g](\Omega)\leqs S_k[h](\Omega).
\label{formu317}
\end{align}
Since $U\geqs u$ in $B_{r_l}$ and $U=u$ on $\partial B_{r_l}$, we get
\begin{align}
S_k[U](B_{3r_l/4})\leqs S_k[U](B_{r_l})\leqs S_k[u](\overline{B}_{r_l}).
\label{formu318}
\end{align}

By $h_{l+1}=U$ on $P_{l+1}$, we have $w_l\leqs h_{l+1}-h_l\leqs0$ on $P_{l+1}$. Together with \eqref{formu316} and \eqref{formu318}, it yields
\begin{align*}
\Vert h_{l+1}-h_l\Vert_{L^\infty(P_{l+1})}\leqs C\Big( \frac{S_k[u](\overline{B}_{r_l})}{r_l^{n-2k}}\Big)^{1/k}.
\end{align*}
Since $h_l,h_{l+1}$ solves \eqref{formu314}, by the comparison principle, we have
\begin{align*}
\Vert h_{l+1}-h_l\Vert_{L^\infty(B_{3r_{l+1}/4})}\leqs C\Big(\frac{S_k[u](\overline{B}_{r_l})}{r_l^{n-2k}}\Big)^{1/k}.
\end{align*}
Then, applying Lemma \ref{lem32} with $r=r_{l+1}/2$, we obtain
\begin{align*}
\operatorname*{Lip}_{B_{r_{l+1}/4}}h_{l+1}\leqs \operatorname*{Lip}_{B_{r_l/4}}h_l+C\Big(\frac{S_k[u](\overline{B}_{r_l})}{r_l^{n-k}}\Big)^{1/k}.
\end{align*}
The coefficient being exactly one is crucial to our proof, which enables us to carry out the iteration. That is,
\begin{align}
\operatorname*{Lip}_{B_{r_l/4}}h_{l}\leqs\operatorname*{Lip}_{B_{r_0/4}}h_{0}+C\sum_{j=0}^{l-1}\Big(\frac{S_k[u](\overline{B}_{r_j})}{r_j^{n-k}}\Big)^{1/k}.
\label{formu319}
\end{align}
\\
\textbf{(3). Estimate on $U$}

Considering $l=p-2$ for \eqref{formu3new1}, we obtain a solution $w_{p-2}$ to the Dirichlet problem
\begin{equation*}
\left\{
\begin{array}{cc}
S_k[w_{p-2}]=S_k[U]  & \text{in }B_{3r_{p-2}/4}, \\
w_{p-2}=0  & \text{on }\partial B_{3r_{p-2}/4}.
\end{array}
\right.
\end{equation*}
By the comparison principle, we have $w_{p-2}+h_{p-2}\leqs U$ in $B_{3r_{p-2}/4}$.

Denote $\widetilde E_{p-1}=B_{7r_{p-1}/8}-B_{5r_{p-1}/8}$. Then 
\begin{align*}
\widetilde E_{p-1}\subset E_{p-1}\quad\text{and}\quad \operatorname{dist}(\widetilde E_{p-1},\partial E_{p-1})=\frac{r_{p-1}}{8}.
\end{align*}
Since $S_k[U]=0$ in $E_{p-1}$, we utilize \eqref{formu35} to obtain
\begin{align*}
|w_{p-2}(x)|\leqs\frac{C}{r_{p-1}^n}\int_{B_{r_{p-1}}}|w_{p-2}|\leqs\frac{C'}{r_{p-2}^n}\Big|\int_{B_{3r_{p-2}/4}}w_{p-2}\Big|
\end{align*}
holds for any $x\in\widetilde{E}_{p-1}$. Together with \eqref{formu310}, it yields that
\begin{align*}
\Vert w_{p-2}\Vert_{L^\infty(\widetilde{E}_{p-1})}\leqs C\Big(\frac{S_k[U](B_{3r_{p-2}/4})}{r_{p-2}^{n-2k}}\Big)^{1/k}.
\end{align*}
Applying \eqref{formu318} with $l=p-2$ and by $w_{p-2}+h_{p-2}\leqs U$, we obtain
\begin{align}
\Vert U-h_{p-2}\Vert_{L^\infty(\widetilde{E}_{p-1})}\leqs C\Big(\frac{S_k[u](\overline{B}_{r_{p-2}})}{r_{p-2}^{n-2k}}\Big)^{1/k}.
\label{formu321}
\end{align}

Next, we prove the gradient estimate for $U$ on $P_{p-1}$. Indeed,
\begin{align*}
P_{p-1}\subset\widetilde E_{p-1}\quad\text{and}\quad \operatorname{dist}(P_{p-1},\partial\widetilde E_{p-1})=\frac{r_{p-1}}{8}.
\end{align*}
Hence, by applying Lemma \ref{lem32} with $r=r_{p-1}/8$ and using \eqref{formu321}, we have for $x\in P_{p-1}$,
\begin{align*}
\operatorname*{Lip}_{B(x,r_{p-1}/16)}U\leqs\operatorname*{Lip}_{B_{r_{p-1}}}h_{p-2}+C\Big(\frac{S_k[u](\overline{B}_{r_{p-2}})}{r_{p-2}^{n-k}}\Big)^{1/k}.
\end{align*}
In particular, taking the supremum over $P_{p-1}$ yields
\begin{align}
\max_{P_{p-1}}|\nabla U|\leqs\operatorname*{Lip}_{B_{r_{p-2}/4}}h_{p-2}+C\Big(\frac{S_k[u](\overline{B}_{r_{p-2}})}{r_{p-2}^{n-k}}\Big)^{1/k}.
\label{formu322}
\end{align}
\\
\textbf{(4). Estimate on $u$}

By our choice of $p$, $R2^{-2p}<r\leqs R2^{-2p+1}=r_{p-1}/2$. Therefore, $u=U$ in $B_r\subset B_{r_{p-1}/2}$. By Corollary \ref{cor21}, $U$ satisfies
\begin{align*}
\frac{1}{r_{p-1}^n}\int_{B_{3r_{p-1}/4}}|\nabla U|\leqs C\max_{P_{p-1}}|\nabla U| +C\Big(\frac{S_k[U](B_{3r_{p-1}/4})}{r_{p-1}^{n-k}}\Big)^{1/k}.
\end{align*}
Using $r>r_{p-1}/4$ and \eqref{formu318} applied with $l=p-1$, we obtain
\begin{align}
\frac{1}{r^n}\int_{B_r}|\nabla u|\leqs C\max_{P_{p-1}}|\nabla U| +C\Big(\frac{S_k[u](\overline{B}_{r_{p-1}})}{r_{p-1}^{n-k}}\Big)^{1/k}.
\label{formu323}
\end{align}

Combining \eqref{formu315}, \eqref{formu319}, \eqref{formu322} and \eqref{formu323}, we have
\begin{align*}
\frac{1}{r^n}\int_{B_r}|\nabla u|\leqs\frac{C}{R^n} \int_{B_R}|\nabla u|+C\sum_{j=0}^{p-1}\Big(\frac{S_k[u](\overline{B}_{r_j})}{r_j^{n-k}}\Big)^{1/k}.
\end{align*}
Recalling that $r_j=R2^{-2j}$, we actually prove \eqref{formu311}.
\\[0.8em]
\textbf{Step 2.} Let $u\in\Phi^k(\Omega)$ satisfy $\mu_k(u)=\mu$ in $\Omega$. Let $\varphi$ be a standard mollifier with compact support, and set $\varphi_\varepsilon(\cdot)=\varepsilon^{-n}\varphi(\cdot/\varepsilon)$. Denote the mollification of $u$ by $u_j=u*\varphi_{\varepsilon_j}$, $\varepsilon_j\downarrow0$. By the convexity of $\Gamma_k$, we have $u_j\in\Phi^k(\Omega')\cap C^\infty(\Omega')$ for $\Omega'\Subset\Omega$ and $j$ sufficiently large. The classical subharmonicity of $u$ shows that
\begin{align*}
u_j\downarrow u\quad\text{locally in }\Omega,
\end{align*}
and by the weak continuity of $k$-Hessian measure (see \cite{TW97measureII}),
\begin{align*}
\mu_j=S_k[u_j]\rightharpoonup\mu_k(u)=\mu.
\end{align*}
Consequently, the following holds for any closed set $K\Subset\Omega$,
\begin{align}
\limsup_{j\to\infty}\mu_j(K)\leqs\mu(K).
\label{formu324}
\end{align}

Assume that $B(x_0,3R)\subset\Omega$ and $u_j$ is well defined on $B(x_0,3R)$. By \eqref{formu311}, we have
\begin{align}
\fint_{B(x_0,r)}|\nabla u_j|\leqs C\fint_{B(x_0,R)}|\nabla u_j|+C\sum_{i=0}^{N(r)}\Big(\frac{\mu_j(\overline{B}(x_0,R2^{-i})}{(R2^{-i})^{n-k}}\Big)^{1/k},
\label{formu325}
\end{align}
where $0<r<R$ and $N(r)$ is defined as in \eqref{formu311}. Since $\Phi^k(\Omega)\subset W^{1,q}_{loc}(\Omega)$ for $q<\frac{nk}{n-k}$, we have $u\in W^{1,q}(B(x_0,R))$ and
$Du_j$ converges to $Du$ in $L^q(B(x_0,R))$. Passing the limit in \eqref{formu325} and making use of \eqref{formu38} and \eqref{formu324}, we obtain
\begin{align*}
\fint_{B(x_0,r)}|\nabla u|&\leqs C\fint_{B(x_0,R)}|\nabla u|+C\sum_{i=0}^{N(r)}\Big(\frac{\mu(\overline{B}(x_0,R2^{-i})}{(R2^{-i})^{n-k}}\Big)^{1/k}\\
&\leqs C\fint_{B(x_0,R)}|\nabla u| +C\mathbf{W}_k^\mu(x_0,2R).
\end{align*}
At the Lebesgue point $x_0$, it follows that
\begin{align}
|\nabla u(x_0)|=\lim_{r\to0}\fint_{B(x_0,r)}|\nabla u|\leqs C\fint_{B(x_0,R)}|\nabla u| +C\mathbf{W}_k^\mu(x_0,2R).
\label{formu326}
\end{align}
Applying the Lebesgue differentiation theorem to $Du\in L_{loc}^q(\Omega)$ yields that \eqref{formu326} holds almost everywhere in $\Omega$. If $\mathbf{W}_k^\mu(x,r)$ is finite for $x\in\Omega$ and some $r>0$, then the right-hand side of \eqref{formu326} is locally uniformly bounded. Since $u\in W^{1,q}_{loc}(\Omega)$, we conclude that $u$ is locally Lipschitz in $\Omega$ by a standard mollification argument. This finishes the proof.

\hfill $\square$
\\[0.3em]
\textbf{Proof of Corollary \ref{cor11}.}

It suffices to prove that $\mathbf{W}_k^\mu(x,R)$ is finite for $x\in\Omega$ and for some $R>0$. For $0<r<R$, if
$\mu(B(x,r))\leqs Mr^{n-k+\varepsilon}$ holds for some constants $M>0$ and $\varepsilon>0$, then
\begin{align*}
\mathbf{W}_k^\mu(x,R)=\int_0^R\Big(\frac{\mu(B(x,r))}{r^{n-k}}\Big)^{1/k}\frac{dr}{r}\leqs \int_0^RM^{1/k}r^{\varepsilon/k-1}dr=\frac{k}{\varepsilon}M^{1/k}R^{\varepsilon/k}.
\end{align*}
In particular, if $\mu=fdx$ is a $L^p$ density in $\Omega$ with $p>\frac{n}{k}$, then by H\"older's inequality,
\begin{align*}
\mu(B(x,r))=\int_{B(x,r)}f\leqs\Big(\int_{B(x,r)}f^p\Big)^{1/p}\Big(\int_{B(x,r)}1\Big)^{1-1/p}\leqs Mr^{n-n/p}.
\end{align*}
Since $p>\frac{n}{k}$, then $\mu(B(x,r))\leqs Mr^{n-k+\varepsilon}$ with $\varepsilon=k-\frac{n}{p}>0$. This completes the proof.

\hfill $\square$
\\[-0.7em]

Theorem \ref{thm11} shows that a $k$-subharmonic function is Lipschitz continuous if the Wolff potential of its $k$-Hessian measure is finite. However, the converse does not hold. Indeed, if $u$ is a Lipschitz $k$-subharmonic function in $B_R$, then $\mu=\mu_k(u)$ satisfies
\begin{align}
\mu(B_r)\leqs Mr^{n-k}\quad\text{for}\quad0<r<R/8,
\label{formu327}
\end{align}
see \cite[Theorem 3.5]{Labutin02potential} for the proof. We remark that although this theorem is stated under the assumption $1\leqs k\leqs n/2$, it remains valid for all $1\leqs k\leqs n$ and the proof is completely the same.

Estimate \eqref{formu327} is sharp, as there exists a convex function $u(x)=|x|$, which is Lipschitz in $\mathbb{R}^n$ and solves
\begin{align*}
\mu_k(u)(B_r)=\binom{n}{k}\omega_nr^{n-k}.
\end{align*}
A direct computation yields that $\mathbf{W}_k^\mu(0,R)$ is infinite for any $R>0$. This indicates that the Wolff potential for Lipschitz $k$-subharmonic functions is not necessarily finite.

\section{Interior $W^{2,1+\varepsilon}$ Regularity }
In this section, we prove the interior $W^{2,1+\varepsilon}$ regularity results for the quadratic Hessian equation. By a rescaling argument, we can assume that $u\in\Phi^2(B_4)$ satisfies $\mu_2(u)=fdx$ with $0<\lambda\leqs f\leqs\Lambda<\infty$ in $B_4$. Then, it suffices to prove that $u\in W^{2,1+\varepsilon}(B_1)$ for some constant $\varepsilon=\varepsilon(n,\lambda,\Lambda)\in(0,1)$. By Theorem \ref{thm11}, $u$ is locally Lipschitz in $B_4$.

Let $\{f_j\}\subset C^\infty(B_3)$ be a sequence of the standard mollification of $f$, so that $\lambda\leqs f_j\leqs\Lambda$ and $f_j$ converges to $f$ in $L^p(B_3)$, $1\leqs p<\infty$.
Consider the approximation solution $u_j$, which solves the Dirichlet problem
\begin{equation*}
\left\{
\begin{array}{cc}
S_2[u_j]=f_j & \text{in }B_3, \\
u_j=\varphi_j  & \text{on }\partial B_3,
\end{array}
\right.
\end{equation*}
where $\varphi_j\in C^\infty(\partial B_3)$ converges to $u$ in $C(\partial B_3)$. By the comparison estimate (see \cite{Trudinger97weak}),
\begin{align*}
\Vert u_m-u_l\Vert_{L^\infty(B_3)}\leqs\Vert\varphi_m-\varphi_l\Vert_{L^\infty(\partial B_3)}+C\Vert f_m^{1/2}-f_l^{1/2}\Vert_{L^{p}(B_3)},
\end{align*}
where $p>\max(n/2,2)$. By our choice of $f_j$ and $\varphi_j$, we have $u_j\to u$ uniformly in $\overline{B_3}$. In particular, $u_j$ has a uniform $L^\infty$ bound in $B_3$. Hence, using the interior gradient estimate in Theorem \ref{thm11}, $|Du_j|\leqs K$ holds in $\overline{B_2}$ with a uniform constant $K>0$. Let $W_j$ be the $2$-subharmonic solution of
\begin{equation*}
\left\{
\begin{array}{cc}
S_2[W_j]=\lambda/4 & \text{in }B_2, \\
W_j=u_j  & \text{on }\partial B_2.
\end{array}
\right.
\end{equation*}
By the comparison principle, $u_j\leqs W_j\leqs H_j$, where $H_j$ is the harmonic function in $B_2$ with boundary trace $u_j$. Using the Poisson formula and global H\"older estimates, we have
\begin{align}
0\leqs W_j-u_j\leqs H_j-u_j\leqs C\operatorname{dist}(x,\partial B_2)^\beta,
\label{formu41}
\end{align}
where $0<\beta<1$ and the constant $C$ depends on $\beta$ and $K$.

By the gradient estimate, we can assume that $|DW_j|\leqs K$ in $B_{3/2}$. To obtain a uniform lower bound for $W_j-u_j$, we apply the separation-propagation argument by \cite[\S3.1$\sim$\S3.3]{LW26sigma2}; note that the seed separation and the constructed barrier function only depend on $n$, $\lambda$ and $K$. Thus, there exists a uniform constant $\delta=\delta(n,\lambda,K)>0$ such that
\begin{align}
W_j-u_j\geqs2\delta\quad\text{in }B_1.
\label{formu42}
\end{align}
By \eqref{formu41}, we have $W_j-u_j\leqs\delta/2$ in $B_2-B_{2-l}$, where $l\leqs(\delta/C)^{1/\beta}$ is fixed.

\subsection{Comparison Solution}
We now present a specific comparison solution constructed by \cite{CZZ26sigma2} and \cite{Lyu26W2p}. Denote
\begin{align*}
G(H)=\sqrt{\sigma_2(H)}\ \ (H\in\Gamma_2),\quad m=\sqrt{\lambda},\quad M=\sqrt{\Lambda}.
\end{align*}
Fix a smooth nondecreasing function $\chi$ such that $\chi=0$ on $(-\infty,0]$ and $\chi=1$ on $[1,\infty)$. For a constant $L\geqs1$ to be determined below, define
\begin{align*}
\eta(s)=\chi(L^{-1}\log{s})\ \ (s>0),\quad\eta(0)=0, \quad\tau=e^L.
\end{align*}
Then, $\eta$ is a nondecreasing function on $[0,\infty)$ with $\eta=0$ on $[0,1]$, $\eta=1$ on $[\tau,\infty)$, and
\begin{align}
\sup_{s>0}(|s\eta'(s)|+|s^2\eta''(s)|)\leqs C_{\chi}/L.
\label{formu43}
\end{align}

Let $b,g,W\in C^\infty(\overline{B_2})$ satisfy $m\leqs b\leqs g\leqs M$. For $P>1$ and $t\in[0,1]$, define
\begin{gather*}
h_t(x,z,S)=g(x)-t\big(g(x)-b(x)\big)\eta\big((W-z-\delta)^4_+S/P\big),\\
F_t(x,z,H)=\frac{\sigma_2(H)-h_t(x,z,S)^2}{S},\qquad S=\operatorname{tr}H>0.
\end{gather*}
For brevity, denote $h=h_t$ and $s=(W-z-\delta)^4_+S/P$. Then,
\begin{align*}
m\leqs h\leqs M,\quad h_z\geqs0\quad\text{and}\quad h_S\leqs0.
\end{align*}
Moreover, $h^2-2Shh_S\geqs m^2$ and
\begin{align*}
(h-Sh_S)^2+S^2hh_{SS}&=(h+t(g-b)s\eta'(s))^2-ht(g-b)s^2\eta''(s)\\
&\geqs m^2-M(M-m)C_\chi/L>0,
\end{align*}
where we take $L=L(m,M,C_\chi)$ sufficiently large; thus $\tau=e^L$ only depends on $m,M$.

For completeness, denote $\phi(S)=h(S)^2/S$. By direct computation,
\begin{align*}
D_HF_t[\xi\otimes\xi]=\frac12\Big|\frac{H}{S}-\xi\otimes\xi\Big|^2+\frac{h^2-2Shh_S}{S^2}>0\qquad(|\xi|=1),
\end{align*}
and
\begin{align*}
D_H^2F_t[X,X]&=-\frac{1}{S}\Big|X-\frac{\operatorname{tr}X}{S}H\Big|^2-\phi''(S)(\operatorname{tr}X)^2<0,\\
\phi''(S)&=\frac{2}{S^3}((h-Sh_S)^2+S^2hh_{SS})>0.
\end{align*}
Hence, the operator $F_t$ is elliptic and concave with respect to $H$ on the entire $\{\operatorname{tr}H>0\}$. In particular, the zero set of $F_t$ with positive trace is contained in $\Gamma_2$. Furthermore, $F_t$ satisfies $(F_t)_z=-2hh_z/S\leqs0$.

With the above preparations in place, one can verify the following existence theorem by the method of continuity; consider the continuous path $F_t(x,v_t,D^2v_t)=0$ with the Dirichlet boundary condition $v_t=u$. Note that the cutoff function $\eta$ is locally constant across $W-z-\delta=0$ for every finite positive $S$, and thus $F_t$ is smooth. We omit the proof here and refer to \cite[Proposition 3.1]{Lyu26W2p} and \cite[Theorem 3.1]{CZZ26sigma2} for details.
\begin{thm}\label{thm41}
Let $u,W,g,b\in C^\infty(\overline{B_2})$ satisfy
\begin{align*}
u,W\in&\Phi^2(B_2),\quad G(D^2u)=g,\quad G(D^2W)=m/2,\\
&m\leqs b\leqs g\leqs M,\quad u=W\quad\text{on }\partial B_2.
\end{align*}
Then for every $P>1$, there exists a unique smooth solution $v\in\Phi^2(B_2)$ of
\begin{equation}
\left\{
\begin{array}{lc}
G(D^2v)+(g-b)\eta\big((W-v-\delta)^4_+\Delta v/P\big)=g  & \text{in }B_2, \\
v=u  & \text{on }\partial B_2.
\end{array}
\right.
\label{formu44}
\end{equation}
Moreover, it satisfies
\begin{align*}
u\leqs v\leqs W,\quad m\leqs G(D^2v)\leqs M.
\end{align*}
\end{thm}

Let $w=(|x|^2-4)/2$, then $w$ solves $G(D^2w)=\sqrt{n(n-1)/2}=:c_n$ and $w=0$ on $\partial B_2$. Set $\tilde{w}=(M-m)w/c_n$, then $G(D^2\tilde{w})=M-m$ in $B_2$. By the concavity of $G$, we have $G(D^2\tilde{w}+D^2v)\geqs G(D^2\tilde{w})+G(D^2v)\geqs G(D^2u)$. Since $\tilde{w}+v=u$ on $\partial B_2$, applying the comparison principle yields $\tilde{w}+v\leqs u$ in $B_2$. This implies
\begin{align}
v-u\leqs-\tilde{w}\leqs C_n(M-m)\quad\text{in }B_2.
\label{formu45}
\end{align}

Denote $g_j=\sqrt{f_j}$. Then $m\leqs g_j\leqs M$, and $u_j,W_j\in\Phi^2(B_2)$ satisfies
\begin{align*}
G(D^2u_j)=g_j,\quad G(D^2W_j)=m/2,\quad u_j=W_j\quad\text{on }\partial B_2.
\end{align*}
Therefore, by Theorem \ref{thm41}, the Dirichlet problem \eqref{formu44} admits a smooth $2$-subharmonic solution $v_j$ for $u=u_j$, $g=g_j$, $W=W_j$ and $b=m$. It satisfies
\begin{align*}
u_j\leqs v_j\leqs W_j\quad\text{and}\quad m\leqs G(D^2v_j)\leqs M.
\end{align*}
Since $W_j-u_j\leqs\delta/2$ in $B_2-B_{2-l}$, we have $\rho_j:=W_j-v_j-\delta<0$ in $B_2-B_{2-l}$. This implies that $\rho_{j,+}$ is compactly supported in $B_{2-l}$. Applying the interior gradient estimates in Theorem \ref{thm11}, we obtain a uniform constant $K_0$ such that
\begin{align*}
\sup_{B_{2-l}}\big(|Dv_j|+|DW_j|\big)\leqs K_0.
\end{align*}
Note that the constants $l,\delta,K_0$ are uniform for $u_j$ but may depend on $u$.

Set $\Omega_j=\{\rho_{j,+}>0\}$. Then $\Omega_j\Subset B_{2-l}$. Without loss of generality, assume that $\Omega_j$ is nonempty. On $\Omega_j$, consider the auxiliary function
\begin{align*}
\Psi_j(x,\xi)=\rho_{j}^4(x)\varphi(\frac{1}{2}|Dv_j|^2)\partial_{\xi\xi}v_j,\quad \xi\in\mathbb{S}^{n-1},
\end{align*}
where for $0\leqs t\leqs K_0^2/2$,
\begin{align}
\varphi(t)=\Big(1-\frac{t}{K_0^2}\Big)^{-a},\quad 1\leqs\varphi\leqs 2^a,\quad a=1/8.
\label{formu46}
\end{align}
Suppose that the maximum value of $\Psi_j$ is attained at $x_j\in\Omega_j$ and $\xi=e_1$. We consider two cases separately. If $\Psi_j(x_j,e_1)>2^a\tau P$, then
\begin{align*}
\rho_j^4\Delta v_j\geqs\rho_j^4\partial_{11}v_j\geqs2^{-a}\Psi_j(x_j,e_1)>\tau P\quad\text{at }x_j.
\end{align*}
Hence, $\rho_j^4\Delta v_j/P>\tau$ holds in a neighborhood $\mathcal{O}$ of $x_j$. Recall that $\eta=1$ on $[\tau,\infty)$. Then, $v_j$ satisfies the equation
\begin{align*}
G(D^2v_j)=m\quad \text{in }\mathcal{O}.
\end{align*}
Applying the Pogorelov's estimate by Chou-Wang \cite[(4.4)$\sim$(4.14)]{CW01variational}, we obtain that the maximum value of $\Psi_j$ is bounded by a constant $C$, which depends only on $n,m$ and $K_0$. Thus by \eqref{formu46}, we have $\rho_j^4\Delta v_j\leqs C_0$ in $\Omega_j$ for some constant $C_0=C_0(n,m,K_0)$.

On the other hand, if $\Psi_j(x_j,e_1)\leqs 2^a\tau P$, then by \eqref{formu46}, $\rho_j^4\Delta v_j\leqs \tau_* P$ in $\Omega_j$, where the constant $\tau_*=2^an\tau$ depends only on $n,m,M$. Combining both cases, we have
\begin{align}
\rho_{j,+}^4\Delta v_j\leqs\max(C_0,\tau_*P).
\label{formu47}
\end{align}

\subsection{Estimate on Trace Tails}
The proof strategy in this subsection is inspired by \cite{Lyu26W2p}. For $A,B\in\Gamma_2$, define
\begin{align*}
[A,B]=\frac{1}{2}\big(\operatorname{tr}A \operatorname{tr}B-\langle A,B\rangle\big),
\end{align*}
where $\langle\cdot,\cdot\rangle$ denotes the Frobenius inner product on Sym(n). Denote
\begin{align*}
d_2(A,B)^2=2\big([A,B]-G(A)G(B)\big).
\end{align*}
By \cite[Lemma 3.6]{CZZ26sigma2}, $d_2(A,B)^2\geqs0$. The following integral identity given in \cite[Proposition 3.8]{CZZ26sigma2} is also useful. For smooth $u,v\in\Phi^2(\Omega)$ with $u=v$ on $\partial\Omega$,
\begin{align}
\int_\Omega d_2(D^2u,D^2v)^2+\frac{1}{2}\int_{\partial\Omega}H_{\partial\Omega}[(u-v)_\nu]^2=\int_\Omega\big(G(D^2u)-G(D^2v)\big)^2,
\label{formu48}
\end{align}
where $\nu$ denotes the outward normal of $\partial\Omega$, and $H_{\partial\Omega}$ denotes the sum of the principal curvatures of $\partial\Omega$ with respect to its inner normal. In particular, when $\Omega$ is mean convex, the boundary integral is nonnegative. Indeed, integrating the identity 
\begin{align*}
d_2(A,B)^2=\big(G(A)-G(B)\big)^2-\sigma_2(A-B)
\end{align*}
for $A=D^2u$, $B=D^2v$ on $\Omega$, and using Reilly's formula gives \eqref{formu48}.

Another useful identity is given in \cite[(3.12)]{Lyu26W2p},
\begin{align*}
d_2(A,B)^2=\frac{1}{2}\operatorname{tr}A\operatorname{tr}B \bigg|\frac{A}{\operatorname{tr}A}-\frac{B}{\operatorname{tr}B}\bigg|^2+\bigg(G(A)\sqrt{\frac{\operatorname{tr}B}{\operatorname{tr}A}}-G(B)\sqrt{\frac {\operatorname{tr}A}{\operatorname{tr}B}}\bigg)^2.
\end{align*}
In particular, if $m\leqs G(A),G(B)\leqs M$, we have
\begin{align}
d_2(A,B)^2\geqs m^2\bigg(\frac{\operatorname{tr}B}{\operatorname{tr}A}+\frac{\operatorname{tr}A}{\operatorname{tr}B}\bigg)-2M^2.
\label{formu49}
\end{align}

Now, we are ready to prove the interior $W^{2,1+\varepsilon}$ regularity. 
\begin{proposition}\label{pro41}
Under the assumptions and notations as above, there exists a constant $\varepsilon=\varepsilon(n,m,M)>0$ such that $u\in W^{2,1+\varepsilon}(B_1)$ and
\begin{align*}
\int_{B_1}|D^2u|^{1+\varepsilon}\leqs C,
\end{align*}
where the constant $C$ depends on $n,m,M,\varepsilon, \delta,K_0$ and $C_0$.
\end{proposition}
\begin{proof}
We begin with the approximation solution $u_j$. Denote $\overline{\rho}_j=W_j-u_j-\delta$. By \eqref{formu42}, we have $\overline{\rho}_j\geqs\delta$ in $B_1$. For simplicity, we suppress the subscript $j$ and write $u,v,W,\rho,\overline{\rho},g$ for $u_j,v_j,W_j,\rho_j,\overline{\rho}_j,g_j$, respectively. Denote
\begin{align*}
T=\Delta u,\quad S=\Delta v,\quad U=\overline{\rho}^4_+ \Delta u,\quad V=\rho^4_+\Delta v.
\end{align*}
Since $u\leqs v\leqs W$, we have $\overline{\rho}\geqs\rho$. By \eqref{formu47}, $V\leqs \tau_*P$ for which $P>1$ is larger than $C_0/\tau_*$. Denote $\mathcal{E}_P=\{\eta(V/P)>0\}$. Recall that $\eta=0$ on $[0,1]$, hence $\mathcal{E}_P\subset\{V>P\}$.

Denote $\mathcal{D}=d_2(D^2u,D^2v)^2$. By \eqref{formu48}, we obtain
\begin{align*}
\int_{B_2}\mathcal{D}\leqs\int_{B_2}(g-m)^2\eta^2\leqs (M-m)^2|\mathcal{E}_P|.
\end{align*}
Set $c_*=m^2/4M^2$. On $\mathcal{E}_P\cap\{U\leqs c_*P\}$, we have $S/T\geqs V/U>1/c_*$ and thus $\mathcal{D}\geqs2M^2$ by virtue of \eqref{formu49}. Therefore,
\begin{align*}
|\mathcal{E}_P|\leqs|\{U>c_*P\}|+\frac{(M-m)^2}{2M^2}|\mathcal{E}_P|,
\end{align*}
which yields that, writing $\theta_P=|\{U>c_*P\}|$,
\begin{align}
|\mathcal{E}_P|\leqs2\theta_P\quad\text{and}\quad\int_{B_2}\mathcal{D}\leqs2(M-m)^2\theta_P.
\label{formu410}
\end{align}

Define $z=v-u\geqs0$. The quadratic structure of \(\sigma_2\) yields the exact linearization
\begin{align}
\sigma_2^{ij}\Big(D^2\frac{u+v}{2}\Big)\partial_{ij}z= \sigma_2(D^2v)-\sigma_2(D^2u).
\label{formu411}
\end{align}
Here, $\sigma_2^{ij}(H)=(\partial/\partial h_{ij})\sigma_2(H)$ and $\operatorname{tr}\sigma_2^{ij}(H)=(n-1)\sigma_1(H)$. Then multiplying \eqref{formu411} by $w=(|x|^2-4)/2$, integrating by parts twice and using the variational structure of $\sigma_2$,
\begin{align*}
\int_{B_2}\operatorname{tr}\sigma_2^{ij}\Big(D^2\frac{u+v}{2}\Big)z=\int_{B_2}w\big(\sigma_2(D^2v)-\sigma_2(D^2u)\big),
\end{align*}
since $D^2w=I_n$ and $w=z=0$ on $\partial B_2$. This implies
\begin{align*}
\int_{B_2}\frac{n-1}{2}(\Delta u+\Delta v)z\leqs\int_{B_2}(-w)2M^2\eta.
\end{align*}
By $\mathcal{E}_P=\{\eta>0\}$ and \eqref{formu410}, we obtain
\begin{align}
\int_{B_2}z\Delta u\leqs C_nM^2|\mathcal{E}_P|\leqs C_nM^2\theta_P.
\label{formu412}
\end{align}

On $\{\overline{\rho}>2z\}$, it follows $\rho=\overline{\rho}-z>\overline{\rho}/2$. Since $V\leqs\tau_*P$, we have
\begin{align}
\frac{T}{S}\geqs\frac{U}{2^4V}\geqs\frac{U}{2^4\tau_*P}.
\label{formu413}
\end{align}
Set $L_*=2^6\tau_*M^2/m^2$. On $\{U>L_*P\}\cap\{\overline{\rho}>2z\}$, by \eqref{formu49} and \eqref{formu413},
\begin{align*}
\mathcal{D}\geqs\frac{m^2U}{2^4\tau_*P}-2M^2\geqs\frac{m^2U}{2^5\tau_*P},
\end{align*}
and hence $U\leqs (2^5\tau_*/{m^2})P\mathcal{D}$. Together with $\eqref{formu410}$, it follows
\begin{align}
\int_{\{U>L_*P\}\cap\{\overline{\rho}>2z\}}U\leqs L_*P\theta_P.
\label{formu414}
\end{align}
Moreover, by \eqref{formu45} and \eqref{formu412},
\begin{align}
\int_{\{0<\overline{\rho}\leqs2z\}}U\leqs\int_{B_2}(2z)^4\Delta u\leqs C_nM^5\theta_P.
\label{formu415}
\end{align}
Let $P>A_0:=\max(C_0/\tau_*,C_nM^5)$. Combining \eqref{formu414} and \eqref{formu415}, we have
\begin{align}
\int_{\{U>L_*P\}}U\leqs(L_*+1)P|\{U>c_*P\}|.
\label{formu416}
\end{align}
Note that $c_*,\tau_*,L_*$ depend only on $n,m,M$; $A_0$ is uniform for $u_j$.

Multiply \eqref{formu416} by $P^{\varepsilon-1}$ with $\varepsilon>0$ to be determined, and integrate over $[A_0,\infty)$. By Tonelli's theorem,
\begin{align}
\int_{A_0}^\infty P^{\varepsilon-1}\int_{\{U>L_*P\}}U\,dP &=\frac{1}{\varepsilon}\int_{\{U>L_*A_0\}}U\bigg[\Big( \frac{U}{L_*}\Big)^\varepsilon-A_0^\varepsilon\bigg],
\label{formu417}\\
\int_{A_0}^\infty P^\varepsilon|\{U>c_*P\}|\,dP&\leqs \frac{1}{(1+\varepsilon)c_*^{1+\varepsilon}}\int_{B_2} U^{1+\varepsilon}.
\label{formu418}
\end{align}
For the remaining part,
\begin{align}
\int_{\{U\leqs L_*A_0\}}U\Big(\frac{U}{L_*} \Big)^ \varepsilon\leqs\int_{\{U\leqs L_*A_0\}}UA_0^\varepsilon.
\label{formu419}
\end{align}
Combining \eqref{formu416}$\sim$\eqref{formu419}, we obtain
\begin{align*}
\int_{B_2}U^{1+\varepsilon}\leqs L_*^\varepsilon A_0^\varepsilon\int_{B_2}U
+\frac{\varepsilon L_*^\varepsilon(L_*+1)}{(1+\varepsilon) c_*^{1+\varepsilon}}\int_{B_2}U^{1+\varepsilon}.
\end{align*}
Choose $\varepsilon=\varepsilon(n,m,M)>0$ sufficiently small, so that the coefficient of the last term is less than $1/2$. Absorbing the last term and restoring the subscript $j$, we have
\begin{align*}
\int_{B_2}\big(\overline{\rho}_{j,+}^4\Delta u_j \big)^{1+\varepsilon}\leqs 2L_*^\varepsilon A_0^\varepsilon\int_{B_2}\overline{\rho}_{j,+}^4\Delta u_j.
\end{align*}
Since $u_j$ has a uniform bound in $B_3$, $\overline{\rho}_{j,+}$ is uniformly bounded. For a positive smooth function $\zeta\in C_c^\infty(B_3)$ which equals one in $B_2$, integrating by parts twice yields a uniform bound for the integral of $\Delta u_j$,
\begin{align*}
\int_{B_2}\Delta u_j\leqs\int_{B_3}\zeta\Delta u_j= \int_{B_3}u_j\Delta\zeta\leqs\sup_{B_3}|u_j|\int_{B_3}|\Delta\zeta|\leqs C(K_0).
\end{align*}
Recall that $\overline{\rho}_j\geqs\delta$ in $B_1$. Combining the above estimates, we arrive at $\int_{B_1}(\Delta u_j)^{1+\varepsilon}\leqs C$ uniformly in $j$; the constant $C$ depends on $n,m,M,\varepsilon,\delta,K_0$ and $C_0$.

Since $u_j\in\Phi^2(B_2)$, $|D^2u_j|\leqs\Delta u_j$. Accordingly, we have $\int_{B_1}|D^2u_j|^{1+\varepsilon}\leqs C$. Therefore, $D^2u_j$ converges weakly in $L^{1+\varepsilon}(B_1)$, up to a subsequence. Since $u_j\to u$ uniformly in $B_3$, the weak limit of $D^2u_j$ coincides with $D^2u$ in the sense of distribution. By the weak lower semi-continuity,
\begin{align*}
\int_{B_1}|D^2u|^{1+\varepsilon}\leqs\liminf_{j\to\infty} \int_{B_1}|D^2u_j|^{1+\varepsilon}\leqs C.
\end{align*}
We conclude that $u\in W^{2,1+\varepsilon}(B_1)$ with a constant $\varepsilon=\varepsilon(n,m,M)>0$.
\end{proof}

\section{Interior $C^{1,\alpha}$ Regularity}
We now establish the interior $C^{1,\alpha}$ regularity for the quadratic Hessian equation. By a rescaling argument, we only need to deal with solutions on $B_4$. We first prove for smooth solutions by contradiction. The proof is inspired by the large-trace propagation argument in \cite[Section 2]{CZZ26sigma2}.
\subsection{The Functions Under Study}
\begin{lemma}\label{lem51}
Let $u_j$ be smooth $2$-subharmonic solutions of
\begin{align*}
S_2[u_j]=f_j\quad\text{in }B_3,
\end{align*}
such that $0<f_j\leqs\Lambda<\infty$, $\Vert u_j\Vert_{L^\infty(B_3)}\leqs C_0$ and $[Du_j]_{C^\alpha(B_1)}\to\infty$ for some $\alpha\in(0,1)$. Here, $\Lambda$ and $C_0$ are uniform constants. For $x\in \overline{B}_2$, define
\begin{align}
P_j(x)=\max_{y\in\overline{B}_2}\frac{|Du_j(y)-Du_j(x)|}{|y-x|^\alpha}.
\label{formu51}
\end{align}
Since $u_j$ is smooth, the maximum value of $P_j(x)$ can be obtained at some point $y\neq x$. Then up to a subsequence, there exist $x_j\in B_{3/2}$ and $M_j\to\infty$, such that
\begin{align}
P_j(x_j)=M_j,\quad0<P_j\leqs2M_j\quad\text{on }B_{R_j}(x_j),
\label{formu52}
\end{align}
where $R_j=M_j^{-\beta}$ with $\beta=(n-1)^{-1}$. Let $y_j$ be the maximizing point of $P_j(x_j)$ in \eqref{formu51}, and denote $\hat{h}_j=|y_j-x_j|$, then
\begin{align*}
\hat{h}_j/R_j\to0.
\end{align*}
\end{lemma}
\begin{proof}
Select $z_j\in B_1$ such that $L_j:=P_j(z_j)\to\infty$. Set $x_{j,0}=z_j$ and $M_{j,0}=L_j$. By induction, we stop at $k\in\mathbb{N}$ if
\begin{align*}
P_j\leqs2M_{j,k}\quad\text{on }B_{M_{j,k}^{-\beta}}(x_{j,k}).
\end{align*}
Otherwise, we choose
\begin{align*}
x_{j,k+1}\in B_{M_{j,k}^{-\beta}}(x_{j,k})\quad\text{and} \quad M_{j,k+1}:=P_j(x_{j,k+1})>2M_{j,k}.
\end{align*}
During the construction, we have $M_{j,k}>2^kM_{j,0}=2^kL_j$, and
\begin{align*}
|x_{j,k}-z_j|\leqs\sum_{l=0}^kM_{j,l}^{-\beta}\leqs \sum_{l\geqs0}(2^lL_j)^{-\beta}=\frac{L_j^{-\beta}}{1-2^{-\beta}}=o(1).
\end{align*}
Thus, after dropping finitely many indices $j$, every point used in the construction belongs to $B_{3/2}$. Moreover, since $u_j$ is smooth, the construction must terminate; otherwise the $C^{1,\alpha}$-norm at $x_{j,k}$ in $B_2$, which equals to $M_{j,k}$, tends to $\infty$ as $k\to\infty$. Denote by $k_j$ the terminal index, and set
\begin{align*}
x_j=x_{j,k_j},\quad M_j=M_{j,k_j}\quad\text{and}\quad R_j=M_j^{-\beta}.
\end{align*}
Hence, $M_j\geqs L_j\to\infty$. This gives \eqref{formu52}.

Moreover, by Theorem \ref{thm11}, we have $|Du_j|\leqs K$ in $B_2$, where the constant $K$ depends only on $n,\Lambda$ and $C_0$. Therefore,
\begin{align*}
M_j=\frac{|Du_j(y_j)-Du_j(x_j)|}{|y_j-x_j|^\alpha}\leqs \frac{2K}{\hat{h}_j^\alpha},
\end{align*}
which yields $\hat{h}_j\leqs(2K/M_j)^{1/\alpha}$. Since $\beta<1/\alpha$, we have $\hat{h}_j/R_j\to0$.
\end{proof}

We next derive a uniform positive second-order difference in a direction by a compactness argument. For a unit vector $e\in\mathbb{S}^{n-1}$, denote
\begin{align*}
Q_h^eu(x)=u(x+he)+u(x-he)-2u(x).
\end{align*}
We write $a\sim b$, if there exist positive constants $C_1,C_2$ such that $C_1a\leqs b\leqs C_2a$.
\begin{lemma}\label{lem52}
Under the hypotheses and notation of Lemma \ref{lem51}, there exist $h_j\sim\hat{h}_j$, points $\omega_j\in B_{R_j/2}(x_j)$ and directions $e_j\in\mathbb{S}^{n-1}$, such that
\begin{align}
Q_{h_j}^{e_j}u_j(\omega_j)\geqs cM_jh_j^{1+\alpha},
\label{formu5new1}
\end{align}
where the constant $c>0$ depends on $n$ and $\alpha$.
\end{lemma}
\begin{proof}
For $z\in\overline{B}_4$, set 
\begin{align}
\psi_j(z)=\frac{u_j(x_j+\hat{h}_jz)-u_j(x_j)-\hat{h}_j Du_j(x_j)\cdot z}{M_j\hat{h}_j^{1+\alpha}}.
\label{formu53}
\end{align}
Since $\hat{h}_j=o(R_j)$, we can assume $4\hat{h}_j<R_j$ and $\psi_j$ is well-defined. Then, it satisfies
\begin{align*}
D\psi_j(z)=\frac{Du_j(x_j+\hat{h}_jz)-Du_j(x_j)}{M_j\hat{h}_j^\alpha}.
\end{align*}
Hence, by direct computation, for $\xi_j=(y_j-x_j)/\hat{h}_j$,
\begin{align}
\psi_j(0)=0,\quad D\psi_j(0)=0,\quad |D\psi_j(\xi_j)|=1\quad\text{and}\quad [D\psi_j]_{C^{\alpha}(\overline{B}_4)}\leqs 2.
\label{formu54}
\end{align}
Claim: there exists a constant $c=c(n,\alpha)>0$, such that after passing to a subsequence,
\begin{align}
\max_{|z|\leqs2,|\eta|\leqs1}[\psi_j(z+\eta)+\psi_j(z-\eta)-2\psi_j(z)]\geqs c.
\label{formu55}
\end{align}
Suppose on the contrary that for all $z\in\overline{B}_2, \eta\in\overline{B}_1$, \eqref{formu55} does not hold uniformly for any $c>0$. Since $\psi_j$ has a uniform $C^{1,\alpha}$ bound, a subsequence of $\psi_j$ converges in $C^1(\overline{B}_3)$ to a function $\psi_\infty$. Moreover, it solves
\begin{align*}
\psi_\infty(z+\eta)+\psi_\infty(z-\eta)-2\psi_\infty(z) \leqs0 \quad\text{for all }z\in\overline{B}_2, \eta\in\overline{B}_1.
\end{align*}
Thus, $\psi_\infty$ is concave in $B_2$. On the other hand, $\psi_j\in\Phi^2(B_4)$ is subharmonic. Hence, by passing to the limit, $\psi_\infty$ is subharmonic in the sense of distribution. This implies, the distributional Hessian of $\psi_\infty$ is negative semidefinite, while its distributional Laplacian is nonnegative. Consequently, $\psi_\infty$ is an affine function in $B_2$. This contradicts $D\psi_\infty(0)=0$ and $|D\psi_\infty(\xi_\infty)|=1$, where $\xi_\infty$ is a limit of $\xi_j$ on $\partial B_1$.

Therefore, we obtain $z_j\in\overline{B}_2$ and $\eta_j\in\overline{B}_1$, such that for every $j$,
\begin{align}
\psi_j(z_j+\eta_j)+\psi_j(z_j-\eta_j)-2\psi_j(z_j)\geqs c.
\label{formu56}
\end{align}
By \eqref{formu54}, the uniform gradient seminorm bound gives 
\begin{align*}
|\psi_j(z+\eta)+\psi_j(z-\eta)-2\psi_j(z)|&=\bigg|\int_0^1 \big(D\psi_j(z+t\eta)-D\psi_j(z-t\eta)\big)\cdot\eta dt\bigg|\\
&\leqs\int_0^12|2t\eta|^{\alpha}|\eta|dt\\
&=\frac{2^{1+\alpha}}{1+\alpha}|\eta|^{1+\alpha}.
\end{align*}
Thus, $|\eta_j|\geqs c_*$ for some $c_*=c_*(\alpha,c)>0$. Substituting \eqref{formu53} into estimate \eqref{formu56} yields
\begin{align*}
u_j(\omega_j+\hat{h}_j\eta_j)+u_j(\omega_j-\hat{h}_j \eta_j)-2u_j(\omega_j)\geqs cM_j\hat{h}_j^{1+\alpha},
\end{align*}
where $\omega_j=x_j+\hat{h}_jz_j\in B_{R_j/2}(x_j)$. Let $e_j=\eta_j/|\eta_j|\in\mathbb{S}^{n-1}$ and $h_j=\hat{h}_j|\eta_j|$. By $|\eta_j|\geqs c_*$, we conclude that $h_j\sim\hat{h}_j$ and $Q_{h_j}^{e_j}u_j(\omega_j)\geqs cM_jh_j^{1+\alpha}$.
\end{proof}

\subsection{Propagation of Large Trace}
For brevity, we suppress the subscript $j$ and translate $\omega_j$ to the origin. Accordingly, let $u,M,R,h,x_M,\omega_M,e$ stand for $u_j,M_j,R_j,h_j,x_j,\omega_j,e_j$, respectively. The translated domain we treat below is 
\begin{align*}
\Omega^h:=\{x:B_{h}(x)+[-h,h]e\subset B_R(x_M)\}.
\end{align*}
Fix a smooth positive radial mollifier $\varphi$ supported in $B_1$, and set $\varphi_{\varepsilon}(\cdot)=\varepsilon^{-n}\varphi(\cdot/\varepsilon)$. For $0<\delta<1$ to be chosen independently of the sequence, let
\begin{align*}
v=u*\varphi_{\delta h}\quad\text{in }\Omega^h.
\end{align*}
Since $[Du]_{C^\alpha(B_R(x_M))}\leqs2M$, we have
\begin{equation}
\begin{gathered}
|v-u|\leqs C_{\varphi}M(\delta h)^{1+\alpha},\quad |Dv-Du|\leqs C_{\varphi}M(\delta h)^\alpha,\\
|D^2v|\leqs C_{\varphi}M(\delta h)^{\alpha-1}.
\end{gathered}
\label{formu57}
\end{equation}
Indeed, the first and second estimates follow by averaging the zero-order and first-order difference of $u$ and using the gradient seminorm bound. For the last estimate,
\begin{align*}
|D^2v(x)|&=\bigg|\int\varphi(y)D_x^2u(x-\delta hy)dy\bigg|\\
&=(\delta h)^{-1}\bigg|\int D_y\varphi(y)D_xu(x-\delta hy)dy\bigg|\\
&=(\delta h)^{-1}\bigg|\int D_y\varphi(y)\big(D_xu(x-\delta hy)-D_xu(x)\big)dy\bigg|\\
&\leqs C_\varphi M(\delta h)^{\alpha-1},
\end{align*}
where we make use of integration by parts and $\int D\varphi=0$ in the equality.

Choose $\delta>0$ sufficiently small so that
\begin{align}
C_\varphi\delta^{1+\alpha}\leqs c/32,
\label{formu59}
\end{align}
where $c$ denotes the constant in \eqref{formu5new1}. Together with \eqref{formu5new1} and \eqref{formu57},
\begin{align}
Q_h^ev(0)=v(he)+v(-he)-2v(0)\geqs 7cMh^{1+\alpha}/8.
\label{formu510}
\end{align}
For convenience, we denote $a=cMh^{1+\alpha}$ and $\Theta=Mh^{\alpha-1}$. For a fixed sufficiently small $\theta\in(0,1/4)$, define
\begin{gather*}
V_T=\int_{-1}^1(1-|t|)v(x+the)dt,\quad Z_T=V_T-u,\\
V_\mu=\frac{1-\theta}{2}\big(v(x+he)+v(x-he)\big)+\theta V_T,\quad Z_\mu=V_\mu-u.
\end{gather*}
Since convolution preserves the local gradient seminorm, it follows that $[Dv]_{C^\alpha}\leqs2M$. Accordingly, $v$ satisfies
\begin{gather*}
|V_T-v|+|V_\mu-v|\leqs CMh^{1+\alpha},\\
|D(V_T-v)|+|D(V_\mu-v)|\leqs CMh^\alpha.
\end{gather*}
Together with \eqref{formu57}, we have $|Z_T|+|Z_\mu|\leqs C'Mh^{1+\alpha}$ and $|DZ_T|+|DZ_\mu|\leqs C'Mh^{\alpha}$. Fixing $\theta\in(0,1/4)$ small enough so that $\theta C'\leqs c/32$, we obtain by \eqref{formu57}$\sim$\eqref{formu510},
\begin{align}
Z_\mu(0)=\frac{1-\theta}{2}\big(v(he)+v(-he)-2u(0)\big) +\theta Z_T(0)\geqs a/4.
\label{formu511}
\end{align}
Set $a_0=a/8$. Let $\Sigma$ be the relative component of $\{Z_\mu>a_0\}$ containing the origin, and define $w=(Z_\mu-a_0)\textbf{1}_\Sigma$. Taking the positive part preserves the Lipschitz bound $C'Mh^\alpha$ for $w$ in $\Omega^h$. Furthermore, on $\{w>0\}$,
\begin{align*}
\frac{1-\theta}{2}Q_h^ev(x)&=Z_\mu-(1-\theta)(v-u)-\theta Z_T\\
&\geqs a_0-C_\varphi M(\delta h)^{1+\alpha}-\theta C'Mh^{1+\alpha}\\
&\geqs a/16.
\end{align*}
Hence, $Q_h^ev\geqs a/8$ on $\{w>0\}$.

By the convexity of the G\aa rding cone $\Gamma_2$, we have $v,V_T,V_\mu\in\Phi^2$. For $H\in\Gamma_2$, denote
\begin{align*}
\mathcal{S}(H)=(\operatorname{tr}H)I-H.
\end{align*}
We have $\mathcal{S}(H)=D\sigma_2(H)\geqs0$. Obviously, $\mathcal{S}(H)\leqs2(\operatorname{tr}H)I$.
\begin{lemma}[Ellipticity at upper contacts]\label{lem53}
On $\{w>0\}$, define
\begin{align*}
H_\mu=D^2V_\mu,\quad K=D^2u,\quad B_\mu=\mathcal{S}\Big(\frac{H_\mu+K}{2}\Big).
\end{align*}
Then,
\begin{align}
\langle B_\mu,D^2w\rangle\geqs-\Lambda.
\label{formu512}
\end{align}
Moreover, there exist constants $\tau,c_0,C_0>0$ such that if
\begin{align}
K\geqs H_\mu-\tau\Theta I,
\label{formu513}
\end{align}
then writing $q=\operatorname{tr}K$, we have
\begin{align}
c_0(q+\Theta)I\big|_{e^\perp}\leqs B_\mu\big|_{e^\perp} \leqs C_0(q+\Theta)I\big|_{e^\perp},\quad 0<B_\mu\leqs C_0(q+\Theta)I.
\label{formu514}
\end{align}
\end{lemma}
\begin{proof}
Quadratic polarization gives the exact linearization
\begin{align*}
\langle B_\mu,D^2w\rangle=\sigma_2(H_\mu)-\sigma_2(D^2u)\geqs-f,
\end{align*}
and \eqref{formu512} follows from $0<f\leqs\Lambda$.

Set $H_T=D^2V_T$. Recall that $a=cMh^{1+\alpha},\Theta=Mh^{\alpha-1}$. Since $Q_h^ev\geqs a/8$ on $\{w>0\}$, we have
\begin{align}
e^TH_Te=\frac{v(x+he)+v(x-he)-2v(x)}{h^2}\geqs\frac{a/8}{h^2}=c_*\Theta.
\label{formu515}
\end{align}
On the other hand, by \eqref{formu57} we have $|D^2v|\leqs C\Theta$. Consequently, $\operatorname{tr}H_T+\operatorname{tr}H_\mu\leqs C\Theta$. For any unit vector $\xi\perp e$, we have by $H_T\in\Gamma_2$,
\begin{align*}
(\xi^TH_T\xi)^2+(e^TH_Te)^2\leqs|H_T|^2\leqs (\operatorname{tr}H_T)^2.
\end{align*}
Denote $q_T=\operatorname{tr}H_T$ and $b_T=e^TH_Te$.
Then, by \eqref{formu515}
\begin{align*}
\xi^T\mathcal{S}(H_T)\xi\geqs q_T-\sqrt{q_T^2-b_T^2}= \frac{b_T^2}{q_T+\sqrt{q_T^2-b_T^2}}\geqs\frac{c_*^2}{2C}\Theta.
\end{align*}
Since $\mathcal{S}(\cdot)$ is linear and positive semidefinite on $\Gamma_2$, and $H_T$ occurs with a fixed positive weight $\theta$ in $H_\mu$, we obtain $\mathcal{S}(H_\mu)\big|_{e^\perp}\geqs c_1\Theta I\big|_{e^\perp}$ for some constant $c_1>0$.

The upper bound of \eqref{formu514} follows directly from the linearity of $\mathcal{S}$ and $\operatorname{tr}H_\mu\leqs C\Theta$. If $q\leqs L_0\Theta$ with some constant $L_0$ to be determined, then the lower bound follows from $B_\mu\geqs\mathcal{S}(H_\mu)/2$ and $\mathcal{S}(H_\mu)\big|_{e^\perp}\geqs c_1\Theta I\big|_{e^\perp}$, by letting $c_0=c_1/2(L_0+1)$.

Next, we assume $q>L_0\Theta$. Since $|H_\mu|\leqs\operatorname{tr}H_\mu\leqs C\Theta$, we obtain by \eqref{formu513},
\begin{align}
\lambda_i(K)\geqs-(C+\tau)\Theta.
\label{formu516}
\end{align}
Order the eigenvalues as $\lambda_1\geqs\cdots\geqs\lambda_n$. Then $\lambda_1\geqs q/n$ and $q-\lambda_1=\lambda_2+\cdots+\lambda_n\geqs0$. In particular, $\lambda_2\geqs0$. Expanding $\sigma_2(\lambda)$ and using \eqref{formu516},
\begin{align*}
\sigma_2(K)=\lambda_1\lambda_2+\sum_{i+j>3}\lambda_i\lambda_j\geqs\lambda_1\lambda_2-Cq\Theta.
\end{align*}
Since $\sigma_2(K)\leqs\Lambda$, $\lambda_1\geqs q/n$ and $\Theta>1$, we obtain $\lambda_2\leqs C\Theta$. Hence, $|\lambda_i|\leqs C\Theta$ for $i\geqs2$. Expanding $\sigma_2(\lambda)$ again yields
\begin{align*}
\lambda_1(q-\lambda_1)=\sigma_2(K)-\sigma_2(\lambda_2, \ldots,\lambda_n)\leqs C\Theta^2,
\end{align*}
which implies $q-\lambda_1\leqs C\Theta^2/q$.

Let $\nu$ be a unit eigenvector of $K$ for $\lambda_1$. Taking the trace of $H_\mu\leqs K+\tau\Theta I$ on $\nu^\perp$ shows that
\begin{align}
\nu^T\mathcal{S}(H_\mu)\nu=\operatorname{tr} \big|_{\nu^\perp}H_\mu\leqs(q-\lambda_1)+(n-1)\tau\Theta \leqs C\Theta^2/q+(n-1)\tau\Theta.
\label{formu517}
\end{align}
On the other hand, $\mathcal{S}(H_\mu)\big|_{e^\perp}\geqs c_1\Theta I\big|_{e^\perp}$ and $\mathcal{S}(H_\mu)\leqs C\Theta I$. We claim that there exists a positive number $\kappa_0$, such that if $|\nu\cdot e|\leqs\kappa_0$, then $\nu^T\mathcal{S}(H_\mu)\nu\geqs c_1\Theta/2$. Indeed, by writing $\nu=\nu_\perp+(\nu\cdot e)e$, we have $|\nu_\perp|^2\geqs1-\kappa_0^2$ and
\begin{align*}
\nu^T\mathcal{S}(H_\mu)\nu&=\nu_\perp^T\mathcal{S}(H_\mu)\nu_\perp+2(\nu\cdot e)e^T\mathcal{S}(H_\mu)\nu_\perp+(\nu\cdot e)^2e^T\mathcal{S}(H_\mu)e\\
&\geqs(1-\kappa_0^2)c_1\Theta-2\kappa_0C\Theta.
\end{align*}
The claim follows by taking $\kappa_0>0$ small. Let $L_0$ be sufficiently large and $\tau$ be sufficiently small, so that the right side of \eqref{formu517} is less than $c_1\Theta/2$. This forces $|\nu\cdot e|>\kappa_0$.

For any $\xi\perp e$, write $\xi=\xi_\perp+(\xi\cdot\nu)\nu$. Since $|\xi\cdot\nu|^2+|e\cdot\nu|^2\leqs1$, we have $|\xi\cdot\nu|^2\leqs1-\kappa_0^2$. By $|\lambda_i(K)|\leqs C\Theta$ for $i\geqs2$, we have $|K\xi_\perp|\leqs C\Theta$ and hence,
\begin{align*}
\xi^T\mathcal{S}(K)\xi&=\operatorname{tr}K-\xi^TK\xi\\
&=q-(\xi\cdot\nu)^2\nu^TK\nu-2(\xi\cdot\nu)\nu^TK\xi_\perp-\xi_\perp^TK\xi_\perp\\
&\geqs q-(1-\kappa_0^2)\lambda_1-C\Theta\geqs\kappa_0^2q-C\Theta.
\end{align*}
Since $q>L_0\Theta$, the last term is larger than $\kappa_0^2q/2$ by enlarging $L_0$ if necessary. Therefore, $\mathcal{S}(K)\big|_{e^\perp}\geqs c_2qI\big|_{e^\perp}$ holds for some constant $c_2>0$. This proves the lower bound of \eqref{formu514} in the remaining case.
\end{proof}

Now, we are ready to prove the transverse propagation estimate using the Alexandrov-Bakelman-Pucci estimate in a form that involves only the upper contact points.
\begin{lemma}\label{lem54}
For a fixed small constant $c_\rho>0$, let
\begin{align}
\rho=c_\rho\min(R,\Theta h),
\label{formu518}
\end{align}
such that the cylinder
\begin{align*}
Q_\rho:=(-\rho,\rho)e+B^{e^\perp}_\rho\Subset\Omega^h.
\end{align*}
Then there exists a measurable set $\mathcal{C}\subset Q_\rho$ satisfying
\begin{align}
|\mathcal{C}|\geqs\bar{c}\rho^{n-1}h,\qquad w>0\quad\text{on }\mathcal{C}.
\label{formu519}
\end{align}
\end{lemma}
\begin{proof}
Since $h/R\to0$ and $\omega_M\in B_{R/2}(x_M)$, the constant $c_\rho$ is universal for all large $M$. Take $A_*=4\sup w/a+1$, and define on the transverse ball $B^{e^\perp}_\rho$,
\begin{align*}
\Psi(y)=\max_{|t|\leqs\rho}\bigg(w(te+y)-A_*a\frac{t^2}{\rho^2}\bigg)-A_*a\frac{|y|^2}{\rho^2}-\frac{a}{16}.
\end{align*}
By \eqref{formu511}, we have $w(0)\geqs a/8$ and hence $\Psi(0)\geqs a/16$, whereas $\Psi(y)<0$ for $|y|\geqs\rho/2$. Let $\Gamma$ be the upper concave envelope of $\Psi_+$ on $B^{e^\perp}_\rho$, and let $S$ denote the interior positive contact set $\{\Gamma=\Psi_+\}$. At $y\in S$, $\Psi(y)>0$ and every maximizing point $t_y$ satisfies
\begin{align*}
A_*a\frac{t_y^2}{\rho^2}+A_*a\frac{|y|^2}{\rho^2}+\frac{a}{16}\leqs w(t_ye+y).
\end{align*}
Hence, $t_y$ lies on $(-\rho/2,\rho/2)$; the penalty exceeds $\sup w$ for $|t|\geqs\rho/2$ by our choice of $A_*$. In particular, $w(t_ye+y)\geqs a/16$ on $S$.

We next estimate the Monge-Ampère density of the concave envelope on $S$. At the second-order Alexandrov point $y_0$ of $\Gamma$, we have $D^2\Gamma\leqs0$ and
\begin{align*}
\phi:\quad(t,y)\mapsto\Gamma(y)+A_*a\frac{t^2}{\rho^2} +A_*a\frac{|y|^2}{\rho^2}+\frac{a}{16}
\end{align*}
touches $w$ from above at an interior maximizing point. Since the mixed Hessian block of $\phi$ with respect to $e\otimes e^\perp$ is exactly zero, we infer that
\begin{align}
D^2w\leqs D^2\phi=D^2\Gamma+\frac{2A_*a}{\rho^2}I\leqs \frac{2A_*a}{\rho^2}I.
\label{formu520}
\end{align}
Since $h/R\to0$ and $\Theta=Mh^{\alpha-1}\to\infty$, we can assume $\rho>C_*h$ with arbitrarily large $C_*$. Hence, by $w>0$ at the lifted contact point, we have
\begin{align*}
H_\mu-K=D^2w\leqs\frac{2A_*a}{\rho^2}I\leqs\frac{C}{C_*^2}\Theta I.
\end{align*}
Then \eqref{formu513} follows by taking $C_*$ sufficiently large. Consequently, Lemma \ref{lem53} yields
\begin{align}
c_0I\big|_{e^\perp}\leqs\frac{B_\mu\big|_{e^\perp}}{q+\Theta}\leqs C_0I\big|_{e^\perp},\qquad0<\frac{B_\mu}{q+\Theta}\leqs C_0I.
\label{formu521}
\end{align}
Together with \eqref{formu512} and \eqref{formu520}, we obtain
\begin{align*}
\Big\langle\frac{B_\mu\big|_{e^\perp}}{q+\Theta},D^2\Gamma \Big\rangle\geqs\Big\langle\frac{B_\mu}{q+\Theta},D^2w \Big\rangle-\Big\langle\frac{B_\mu}{q+\Theta},\frac{2A_*a}{\rho^2}I\Big\rangle\geqs-C\Big(\frac{1}{\Theta}+\frac{a}{\rho^2}\Big).
\end{align*}
Since $\rho\leqs c_\rho\Theta h$, then $\Theta^{-1}\leqs c'a/\rho^2$. Thus, the lower ellipticity bound in \eqref{formu521} implies
\begin{align*}
\operatorname{tr}(-D^2\Gamma)\leqs C\frac{a}{\rho^2}.
\end{align*}
Applying the arithmetic–geometric mean inequality, we obtain
\begin{align}
\det(-D^2\Gamma)\leqs C\Big(\frac{a}{\rho^2}\Big)^{n-1} \quad\text{at almost every point of }S.
\label{formu522}
\end{align}

To continue, we apply the ABP estimate for continuous solutions, see \cite[Section 3]{CCnonlinear},
\begin{align}
\sup\Psi_+\leqs C(n)\rho\Big(\int_S\det(-D^2\Gamma)\Big)^{1/(n-1)}.
\label{formu523}
\end{align}
It is valid for the present functions: $w$ is semiconvex, as the positive part of a smooth function; taking the maximum of $\big[w(te+y)-A_*at^2/\rho^2\big]$ over a compact axial interval preserves the semiconvexity with a finite constant. Therefore, $\Psi$ is semiconvex in $B^{e^\perp}_\rho$. Hence, the concave envelope $\Gamma$ is of class $C^{1,1}$, and the ABP estimate \eqref{formu523} holds in the full-measure set satisfying \eqref{formu522}. Recalling $\Psi(0)\geqs a/16$, we conclude that
\begin{align*}
|S|\geqs \hat{c}\rho^{n-1}.
\end{align*}
For every $y\in S$, $w(t_ye+y)\geqs a/16$ for some $|t_y|\leqs\rho/2$. Since $w$ has a Lipschitz bound $C'Mh^\alpha$, its fiber superlevel set $\{w>0\}$ contains an interval with length at least $c'h$. This interval remains in $(-\rho,\rho)e$. Applying Fubini's theorem yields
\begin{align*}
|\{w>0\}\cap Q_\rho|\geqs\bar{c}\rho^{n-1}h,
\end{align*}
where $\bar{c}>0$ is a uniform constant. This finishes the proof of Lemma \ref{lem54}.
\end{proof}
\begin{remark}\label{rmk51}
The proof does not assert uniform ellipticity on $e^\perp$ for non-contact points. In particular, we cannot infer a pointwise bound for $D^2u$ from an averaged Hessian bound. The essential tool is the matrix comparison \eqref{formu513}, which holds at contact points by virtue of the upper test.
\end{remark}

Recall that $v=u*\varphi_{\delta h}$ and
\begin{align*}
V_T=\int_{-1}^1(1-|t|)v(x+the)dt,\quad H_T=\int_{-1}^1(1-|t|)D^2v(x+the)dt.
\end{align*}
Since both the mollifier and the triangular weight have total mass one, we have
\begin{align*}
\operatorname{tr}H_T=\int_{-1}^1(1-|t|)(\Delta u*\varphi_{\delta h})(x+the)dt=:\mathcal{F}_h*\Delta u,
\end{align*}
where $\mathcal{F}_h$ is a positive probability measure compactly supported in $B_h+[-h,h]e$. Specifically, for $x\in\Omega^h$, $\mathcal{F}_h*\Delta u(x)$ depends only on the value of $\Delta u$ in $B_R(x_M)$.

Since $H_T\in\Gamma_2$, we have $\operatorname{tr}H_T\geqs e^TH_Te$. According to \eqref{formu515},
\begin{align}
\mathcal{F}_h*\Delta u=\operatorname{tr}H_T\geqs e^TH_Te\geqs c_*\Theta\quad \text{on }\{w>0\}.
\label{formu524}
\end{align}
In particular, it holds on $\mathcal{C}$, where $|\mathcal{C}|\geqs\bar{c}\rho^{n-1}h$.

\subsection{Completion of the Proof}
We first prove the interior $C^{1,\alpha}$ estimate by contradiction, under the $W^{2,p}$-regularity assumption on solutions.
\begin{proposition}\label{pro51}
Assume that $p>1$ and
\begin{align*}
0<\alpha\leqs\frac{p-1}{n+p-1}.
\end{align*}
Suppose that $u\in\Phi^2(B_3)\cap C^\infty(B_3)$ satisfies
\begin{align*}
0<S_2[u]\leqs\Lambda<\infty,\quad\Vert u\Vert_{L^\infty(B_3)}\leqs C_0,\quad\int_{B_2}(\Delta u)^p\leqs L_0,
\end{align*}
then it holds
\begin{align}
[Du]_{C^\alpha(B_1)}\leqs C(n,p,\alpha,\Lambda,C_0,L_0).
\label{formu525}
\end{align}
\end{proposition}
\begin{proof}
Assume the assertion is false. Then there exist smooth functions $u_j\in\Phi^2(B_4)$, with
\begin{align*}
0<S_2[u_j]\leqs\Lambda<\infty,\quad\Vert u_j\Vert_{L^\infty(B_3)}\leqs C_0,\quad\int_{B_2}(\Delta u_j)^p\leqs L_0,
\end{align*}
but
\begin{align*}
[Du_j]_{C^\alpha(B_1)}\to\infty.
\end{align*}
Combining Lemma \ref{lem51}$\sim$\ref{lem54} and extracting a subsequence, we have
\begin{gather*}
M_j\to\infty,\quad R_j=M_j^{-\beta},\quad\beta=(n-1)^{-1},\\
h_j/R_j\to0,\quad\Theta_j=M_jh_j^{\alpha-1}, \quad\rho_j=c_\rho\min(R_j,\Theta_jh_j),
\end{gather*}
measurable sets $\mathcal{C}_j\subset\Omega^j\Subset B_2$ and positive probability measures $\mathcal{F}_j$, such that
\begin{align}
|\mathcal{C}_j|\geqs\bar{c}\rho_j^{n-1}h_j,\quad \mathcal{F}_j*\Delta u_j\geqs c_*\Theta_j\quad\text{on }\mathcal{C}_j.
\label{formu526}
\end{align}
On the other hand, Jensen's inequality and Fubini's theorem yields
\begin{align}
\int_{\Omega^j}(\mathcal{F}_j*\Delta u_j)^p\leqs\int_{B_2}(\Delta u_j)^p\leqs L_0.
\label{formu527}
\end{align}
Combining \eqref{formu526} and \eqref{formu527}, we obtain $\rho_j^{n-1}h_j\Theta_j^p\leqs CL_0$. We next consider two cases separately. For the case $R_j\geqs\Theta_jh_j$,
\begin{align*}
\rho_j^{n-1}h_j\Theta_j^p=cM_j^{p+n-1}h_j^{\alpha(p+n-1)-(p-1)}.
\end{align*}
The exponent of $h_j$ is nonpositive by our choice of $\alpha$. Hence, this expression tends to $\infty$ as $M_j\to\infty$ and $h_j\to0$. For the other case $R_j<\Theta_jh_j$,
\begin{align*}
\rho_j^{n-1}h_j\Theta_j^p=cM_j^{p-1}h_j^{\alpha p-(p-1)}.
\end{align*}
Similarly, the exponent of $h_j$ is negative, and this expression tends to $\infty$. Both alternatives contradict the boundedness of $\rho_j^{n-1}h_j\Theta_j^p$. This proves \eqref{formu525}.
\end{proof}

Finally, we prove the estimate for general $2$-subharmonic solutions with positive density bounded above and below, by adapting an approximation argument as in Section 4.
\begin{proposition}\label{pro52}
Let $u\in\Phi^2(B_4)$ and $\mu_2(u)=fdx$ with $f\in L^\infty(B_4)$. Suppose that
\begin{align*}
0<\lambda\leqs f\leqs\Lambda<\infty,\quad\Vert u\Vert_{L^\infty(B_4)}\leqs C_0.
\end{align*}
Then there exists a constant $\alpha=\alpha(n,\lambda,\Lambda)>0$ such that $u\in C^{1,\alpha}(B_1)$ and
\begin{align*}
[Du]_{C^\alpha(B_1)}\leqs C(n,\lambda,\Lambda,\alpha,C_0).
\end{align*}
\end{proposition}
\begin{proof}
Let $\{f_j\}\subset C^\infty(B_3)$ be a sequence of mollification of $f$, so that $\lambda\leqs f_j\leqs\Lambda$ and $f_j$ converges to $f$ in $L^p(B_3)$, $1\leqs p<\infty$. As discussed in the beginning of Section 4, the following approximation Dirichlet problem admits a unique smooth solution $u_j$,
\begin{equation*}
\left\{
\begin{array}{cc}
S_2[u_j]=f_j & \text{in }B_3, \\
u_j=\varphi_j  & \text{on }\partial B_3,
\end{array}
\right.
\end{equation*}
where $\varphi_j\in C^\infty(\partial B_3)$ converges to $u$ in $C(\partial B_3)$. Moreover, using the comparison estimate, we have $u_j\to u$ uniformly in $\overline{B_3}$. In particular, $u_j$ has a uniform $L^\infty$ bound in $B_3$.

Applying the interior gradient estimate, we get a uniform Lipschitz bound $K_0$ in $B_{5/2}$. By covering $B_2$ with finitely many balls contained in $B_{5/2}$, we obtain by Proposition \ref{pro41} that there exists a constant $\varepsilon=\varepsilon(n,\lambda,\Lambda)>0$ such that $u_j\in W^{2,1+\varepsilon}(B_2)$ and
\begin{align*}
\int_{B_2}|D^2u_j|^{1+\varepsilon}\leqs L_0,
\end{align*}
where the constant $L_0$ depends on $n,\lambda,\Lambda, \varepsilon$ and $K_0$. Applying Proposition \ref{pro51} yields the uniform gradient seminorm for $u_j$ with a constant $0<\alpha\leqs\varepsilon/(n+\varepsilon)$,
\begin{align*}
[Du_j]_{C^{\alpha}(B_1)}\leqs C(n,\varepsilon, \alpha,\Lambda,C_0,L_0).
\end{align*}
The desired estimate follows by passing $j\to\infty$.
\end{proof}

\section*{Declaration on the use of AI}
During the preparation of this work, ChatGPT was used as an auxiliary tool for exploring mathematical ideas and performing calculations. For the interior gradient estimates, the author considered the $k$-Perron lifting on disjoint shells and reduced the proof to relate the gradient bounds for successive liftings. Then the Lipschitz stability lemma via two-point maximum principle was suggested by GPT. The interior $W^{2,1+\varepsilon}$ estimate was initially proposed with the aid of GPT, and the author simplified the proof and verified its correctness. For the interior $C^{1,\alpha}$ estimate, we were first motivated by the large-trace propagation argument in \cite{CZZ26sigma2}. However, the transverse uniform ellipticity of the linearization equation no longer holds generally, see Remark \ref{rmk51}. This problem was solved with the assistance of GPT, by establishing the uniform ellipticity at upper contact points and using the ABP estimate to achieve the transverse propagation. All mathematical arguments and results were verified by the author, who takes full responsibility for the content of this manuscript.

\small
\bibliographystyle{plain}
\bibliography{cite}

\medskip\medskip
{\em Address and E-mail:}
\medskip

{\em Rongxun He}

{\em School of Mathematical Sciences, Fudan University, Shanghai 200433, China}

{\em rxhe24@m.fudan.edu.cn}

\end{document}